\documentclass[11pt,a4paper]{article}

\usepackage[left=27mm,right=27mm,top=26mm,bottom=28mm]{geometry}
\usepackage{amsmath,amssymb,amsthm,mathtools}
\usepackage{microtype}
\usepackage{enumitem}
\usepackage{hyperref}
\usepackage[nameinlink,capitalize,noabbrev]{cleveref}
\usepackage{lmodern}
\usepackage{titlesec}

\hypersetup{
  colorlinks=true,
  linkcolor=red,
  citecolor=blue,
  urlcolor=blue
}

\setlist[enumerate]{leftmargin=2.35em,itemsep=2pt,topsep=4pt}
\setlist[itemize]{leftmargin=2.1em,itemsep=2pt,topsep=4pt}
\titleformat{\section}{\normalfont\Large\bfseries}{\thesection}{1em}{}
\titlespacing*{\section}{0pt}{2.3ex plus .6ex minus .2ex}{1.1ex}
\allowdisplaybreaks

\newtheorem{theorem}{Theorem}[section]
\newtheorem{lemma}[theorem]{Lemma}
\newtheorem{proposition}[theorem]{Proposition}
\newtheorem{corollary}[theorem]{Corollary}
\newtheorem{conjecture}[theorem]{Conjecture}
\newtheorem{question}[theorem]{Question}
\theoremstyle{definition}
\newtheorem{definition}[theorem]{Definition}
\numberwithin{equation}{section}

\newcommand{\C}{\mathcal C}
\newcommand{\T}{\mathcal T}
\newcommand{\Palt}{\mathcal P}
\newcommand{\Qcal}{\mathcal Q}
\newcommand{\N}{\mathbb N}
\newcommand{\ext}{\mathrm{ext}}
\newcommand{\rev}{\operatorname{rev}}
\newcommand{\Sym}{\mathrm S}
\newcommand{\Sp}{(\spadesuit^{*})}
\newcommand{\Piuk}{\Pi_{\mathrm u}^{(k)}}
\newcommand{\co}{\mathrm{co}}

\title{Coordinate-extension degrees and layered $k$-uniform hypergraphs}
\author{\large Jiabao Yang\thanks{Email: \texttt{jbyang1215@nju.edu.cn}}\\[-1mm]
\normalsize School of Mathematics, Nanjing University, Nanjing 210093, P.R. China}
\date{}

\begin{document}
\maketitle

\begin{abstract}
Let $\Palt=(\C,\T)$ be a $k$-palette. For $0\le t\le k-1$, its $t$th coordinate-extension degree is the minimum, over every choice of $t$ coordinates and every assignment of colors to them, of the proportion of assignments to the remaining $k-t$ coordinates that complete the fixed colors to an admissible $k$-tuple. For a $k$-graph $F$, we define $\pi_t^{\ext}(F)$ as the supremum of this degree over all palettes not admitted by $F$. We prove that
\[
\pi_t^{\ext}(F)=0
\quad\text{if and only if}\quad
F\text{ is }t\text{-layered}.
\]
We also relate $t$-layeredness to vanishing orders, min-layeredness, max-layeredness, and layeredness. 
These results recover and extend
previous characterizations of Reiher, R\"odl, and Schacht and of
Lamaison, and answer a question of Lamaison for $3$-graphs.

At $t=0$, the parameter $\pi_0^{\ext}(F)$ is the $(k-2)$-uniform Tur\'an density $\pi_{k-2}(F)$. For every $k\ge3$ and $r\ge2$, we construct a finite $k$-graph $F_{k,r}$ with
\(
\pi_{k-2}(F_{k,r})=2(r-1)/rk^k.
\)
Thus $2/k^k$ is an accumulation point for single forbidden $k$-graphs. 
We also show that the least density of a $k$-graph that fails condition $\Sp$ of Lin, Wang and Zhou is $4/(3k^k)$. 
Finally, for every admissible matching of size $m$, we construct a $k$-graph that satisfies $\Sp$ for every coordinate pair, has no vanishing order, and has density $2^m/k^k$. 
This disproves a conjecture of Lin, Wang and Zhou for every $k\ge3$. 
\end{abstract}



\section{Introduction}

For a positive integer $k$ and a set $V$, let $\binom{V}{k}$ be the family of all $k$-subsets of $V$. 
A $k$-uniform hypergraph, or a $k$-graph, is a pair $H=(V,E)$ with $E\subseteq\binom{V}{k}$. 
A $k$-graph $H$ is \emph{$F$-free} if it contains no copy of the $k$-graph $F$ as a subgraph. 
The \emph{Tur\'an number} of $F$, denoted by $\operatorname{ex}(n,F)$, is the maximum number of edges in an $F$-free $k$-graph on $n$ vertices. 
Tur\'an-type problems have been a central theme in extremal combinatorics for many decades. 
Since determining $\operatorname{ex}(n,F)$ exactly is often difficult, it is natural to study its asymptotic behaviour. 
The \emph{Tur\'an density} of $F$ is defined by $ \pi(F)=\lim_{n\to\infty} \operatorname{ex}(n,F)/\binom{n}{k}$,
where the existence of the limit follows from a standard averaging argument. 

For graphs, the asymptotic problem is completely understood. 
The classical Erd\H{o}s--Stone--Simonovits theorem~\cite{Er,stone} states that every graph $F$ containing at least one edge satisfies $ \pi(F)=\frac{\chi(F)-2}{\chi(F)-1}$.
In sharp contrast, determining Tur\'an densities of $k$-graphs with $k\ge 3$ is difficult even for very small hypergraphs. 
Let $K_4^{(3)}$ denote the complete $3$-graph on four vertices, and let $K_4^{(3)-}$ be the $3$-graph obtained from $K_4^{(3)}$ by deleting one edge. 
The exact values of both $\pi(K_4^{(3)-})$ and $\pi(K_4^{(3)})$ remain unknown. 
In particular, determining $\pi(K_4^{(3)})$, usually referred to as the \emph{tetrahedron problem}, is one of the main open problems in extremal combinatorics. 
For further background and related results, we refer the reader to~\cite{deCaen1994,Furedi1991,Keevash2011,Sidorenko1995}. 

These difficulties also motivate the study of variants in which the edges of the host hypergraph are required to be distributed uniformly.
Erd\H{o}s and S\'os \cite{ErdosSos} initiated the study of Tur\'an problems under a uniform density condition on large vertex subsets.
Given a $(k-2)$-graph $G^{(k-2)}$, let $K_k(G^{(k-2)})$ be the family of $k$-subsets of $V(G^{(k-2)})$ that span a complete $(k-2)$-graph. 
An $n$-vertex $k$-graph $H=(V,E)$ is $(d,\mu,k-2)$-dense if
$|K_k(G^{(k-2)})\cap E|\ge d|K_k(G^{(k-2)})|-\mu n^k$ for every $(k-2)$-graph $G^{(k-2)}$ with vertex set $V$. 
The $(k-2)$-uniform Tur\'an density of a $k$-graph $F$ is
\begin{align*}
 \pi_{k-2}(F)=\sup\{d\in[0,1]:{}&\text{ for every }\mu>0\text{ and }n_0\in\mathbb N,\text{ there exists an }F\text{-free}\\
 &\text{$(d,\mu,k-2)$-dense $k$-graph $H$ with $|V(H)|\ge n_0$}\}.
\end{align*}
When $k=3$, the edge set of a $1$-graph on $V$ consists of the singletons of some subset $U\subseteq V$, and $K_3(G^{(1)})=\binom U3$. 
Hence $\pi_{k-2}=\pi_1$ is exactly the original uniform Tur\'an density of $3$-graphs. 
For $k\ge4$, the parameter considered here is stronger than the usual $1$-uniform Tur\'an density. 
The latter is a different problem: Lamaison~\cite{LamaisonBeyond} recently proved that, in every uniformity $k\ge3$, the values $1/4$ and $\binom{k}{2}^{-\binom{k}{2}}$ occur as ordinary uniform Tur\'an densities.

The first general exact results were obtained by Reiher, R\"odl, and Schacht~\cite{RRSMantel} in 2016. 
They proved that $\pi_{k-2}(F^{(k)})=2^{1-k}$ for the unique three-edge $k$-graph $F^{(k)}$ on $k+1$ vertices. 
Reiher, R\"odl, and Schacht~\cite{ReiherRodlSchacht2018JLMS} also characterized all $k$-graphs $F$ with $\pi_{k-2}(F)=0$ and deduced that
$\pi_{k-2}(F)=0$ or $\pi_{k-2}(F)\ge k^{-k}$ for every $k$-graph $F$. 
This yields a jump phenomenon at zero.
Garbe, Kr\'al', and Lamaison~\cite{GKLminimum} showed that the lower endpoint $1/27$ is attained when $k=3$. 
Lin, Wang, and Zhou~\cite{LinWangZhou2025} then developed a framework based on hypergraph regularity for general $k$ and proved that the minimum positive value $k^{-k}$ is attained for every $k\ge3$.

The case $k=3$ has a much richer list of exact results. 
Glebov, Kr\'al', and Volec~\cite{GKV}, and independently Reiher, R\"odl, and Schacht~\cite{RRSweak}, proved that the $K_4^{(3)-}$ has uniform Tur\'an density $1/4$. 
Recently, Kielak, Kr\'al', Lamaison, Liu, Shu, and Wu \cite{Kielak2026} showed that the uniform Tur\'an density of $K_4^{(3)}$ is equal to $1/2$; this confirms that R\"{o}dl's \cite{Rodl1986} lower bound construction from 1986 is optimal.
Buci\'c, Cooper, Kr\'al', Mohr, and Munh\'a Correia~\cite{BucicEtAl} determined the density of every tight $3$-uniform cycle of length at least five: the answer is zero when the length is divisible by three and $4/27$ otherwise. 
Garbe, Il'kovi\v{c}, Kr\'al', Ku\v{c}er\'ak, and Lamaison~\cite{GarbeEtAl827} established the value $8/27$. 
Further exact $3$-graph classes and the star problem have also been studied; see, for example,~\cite{ChenSchulke,LamaisonWu,ReiherSurvey}.

Palettes, which extend a construction of R\"odl~\cite{Rodl1986}, provide all known sharp lower bounds in this setting. 
Lamaison~\cite{Lamaison2024} proved that palettes determine the uniform Tur\'an density of every single $3$-graph and obtained the values $(r-1)/(2r)$ for all $r\ge2$. 
Kr\'al', Ku\v{c}er\'ak, Lamaison and Tardos~\cite{KralKucerakLamaisonTardos2025} established a multi-palette classification theorem and obtained $4/81$. 
King, Piga, Sales and Sch\"ulke~\cite{KingPigaSalesSchulke2025} proved that every palette Lagrangian is the density of a finite forbidden family. 
Lin, Sun, Wang and Zhou~\cite{LinSunWangZhou2026} extended the palette equality and classification to $k$-graphs and arbitrary forbidden families. 
In particular, for every $k\ge3$ and $r\ge2$, they realized
\[
\frac{r-1}{r},\quad \frac{(r-1)^2}{r^2},\quad \frac{r-1}{2r},\quad
\frac{(k-1)^k}{k^k},\quad \frac{4(k-2)^{k-2}}{k^k},\quad
\frac{4(k-2)^{k-2}}{3k^k}
\]
as densities of single $k$-graphs and constructed the first non-principal families for this density. 
Recently, Ai, Chen, Chen, Yan and
Zhao~\cite{AiChenChenYanZhao2026} constructed a sequence of single
finite $3$-graphs whose uniform Tur\'an densities converge to $1/3$,
and proved that the exact values $1/3-1/(3m^2)$, $m\ge2$, are attained
by finite families of $3$-graphs. Li, Liu and
Pikhurko~\cite{LiLiuPikhurko2026} proved that the set of uniform
Tur\'an densities of possibly infinite families of $3$-graphs
contains a terminal interval.

\subsection{Palettes, coordinate-extension degree, and \texorpdfstring{$t$}{t}-layered graphs}\label{subsec:coordinate-extension}

Throughout the rest of the paper, $k\ge3$, all hypergraphs are finite, and $\N=\{1,2,\ldots\}$. For $n\in\N$, write $[n]=\{1,\ldots,n\}$. If $F$ is a $k$-graph, its $(k-1)$-shadow is
\[
\partial F=\bigl\{S\in\tbinom{V(F)}{k-1}:S\subseteq e\text{ for some }e\in E(F)\bigr\}.
\]
We first recall the palette framework of Lin, Sun, Wang and Zhou~\cite{LinSunWangZhou2026}.

\begin{definition}[$k$-palette and admission]\label{def:palette}
Fix $k\ge3$. A \emph{$k$-palette} is a pair $\Palt=(\C,\T)$, where $\C$ is a non-empty finite set, whose elements are called colors, and $\T\subseteq\C^k$ is the set of admissible ordered $k$-tuples. Its density is
\[
d(\Palt)=\frac{|\T|}{|\C|^k}.
\]
Let $F$ be a $k$-graph. 
We say that $F$ \emph{admits} $\Palt$ if there are a linear order $\prec$ on $V(F)$ and a coloring $\varphi:\partial F\to\C$ such that the following holds. 
For every edge $e=\{v_1,\ldots,v_k\}$ with $v_1\prec\cdots\prec v_k$, write
\[
p_i(e)=e\setminus\{v_i\}\quad(i\in[k]).
\]
Then
\[
\bigl(\varphi(p_1(e)),\ldots,\varphi(p_k(e))\bigr)\in\T.
\]
The set $p_i(e)$ is the $(k-1)$-subset obtained by deleting the vertex in position $i$. We call it the $i$th facet of $e$ and use this notation throughout.
\end{definition}

The standard random construction explains the role of palettes. Order the vertices of $[n]$ and color their $(k-1)$-subsets independently and uniformly from $\C$. Include an ordered $k$-set as an edge precisely when its facet colors form a tuple in $\T$. If $F$ does not admit $\Palt$, every resulting hypergraph is $F$-free. The concentration argument for this construction~\cite{LinSunWangZhou2026} yields
\begin{equation}\label{eq:palette-lower}
\pi_{k-2}(F)\ge d(\Palt).
\end{equation}

The next definition extends the minimum palette degree and minimum palette codegree introduced by Lamaison~\cite{Lamaison2024} for $3$-graphs.

\begin{definition}[Coordinate-extension degree]\label{def:extension-degree}
Let $\Palt=(\C,\T)$ be a $k$-palette. 
For $I\subseteq[k]$, let $\C^I$ be the set of functions from $I$ to $\C$. For $\boldsymbol c=(c_i)_{i\in I}\in\C^I$, define
\[
N_I^{\Palt}(\boldsymbol c)=\bigl\{(y_j)_{j\notin I}:(y_1,\ldots,y_k)\in\T,\ y_i=c_i\text{ for every }i\in I\bigr\}.
\]
Thus $N_I^{\Palt}(\boldsymbol c)$ consists of the assignments to the remaining coordinates that complete the prescribed colors to an admissible tuple.
For $0\le t\le k-1$, the \emph{$t$th coordinate-extension degree} of $\Palt$ is
\begin{equation}\label{eq:delta-t}
\delta_t(\Palt)=\min_{I\in\binom{[k]}t}\ \min_{\boldsymbol c\in\C^I}
\frac{|N_I^{\Palt}(\boldsymbol c)|}{|\C|^{k-t}}.
\end{equation}
\end{definition}

For $t=0$, the only choices are $I=\varnothing$ and the empty assignment, so $\delta_0(\Palt)=d(\Palt)$. For a $3$-palette, $\delta_1$ and $\delta_2$ are Lamaison's minimum palette degree and minimum palette codegree, respectively~\cite{Lamaison2024}.

For a $k$-graph $F$, define
\begin{equation}\label{eq:pi-ext}
\pi_t^{\ext}(F)=\sup\bigl\{\delta_t(\Palt):\Palt\text{ is a $k$-palette not admitted by }F\bigr\}.
\end{equation}
Here the supremum is taken over all finite non-empty color sets and all sets of admissible tuples. If there is no such palette, the supremum is defined to be zero. Since $\delta_0(\Palt)=d(\Palt)$, the palette theorem of Lin, Sun, Wang and Zhou~\cite[Theorem~1.4]{LinSunWangZhou2026} states that
\[
\pi_0^{\ext}(F)=\pi_{k-2}(F).
\]

We now define the structural class that characterizes the vanishing of $\pi_t^{\ext}$.

\begin{definition}[$t$-layered $k$-graph]\label{def:t-layered}
Let $0\le t\le k-1$. A $k$-graph $F$ is \emph{$t$-layered} if there are a linear order on $V(F)$ and a map $\varphi:\partial F\to\N$ with the following property. The distinct vectors in
\[
\bigl\{(\varphi(p_1(e)),\ldots,\varphi(p_k(e))):e\in E(F)\bigr\}
\]
can be ordered as $\boldsymbol a^1,\ldots,\boldsymbol a^r$, where $\boldsymbol a^s=(a_1^s,\ldots,a_k^s)$, so that for every $s\in[r]$ there is a set $I_s\in\binom{[k]}t$ satisfying
\begin{enumerate}[label=\textup{(D\arabic*)},ref=\textup{(D\arabic*)}]
\item\label{D1} the entries $a_h^s$, $h\notin I_s$, are pairwise distinct;
\item\label{D2} for every $h\notin I_s$,
\[
a_h^s\notin\{a_i^s:i\in I_s\}\cup\bigcup_{u<s}\{a_1^u,\ldots,a_k^u\}.
\]
\end{enumerate}
An edgeless $k$-graph is $t$-layered. We call the coordinates in $I_s$ the \emph{tail coordinates}, and the remaining coordinates the \emph{head coordinates}.
\end{definition}

Thus each vector has $k-t$ distinct head entries that occur neither among its tail entries nor in an earlier vector. Tail entries may coincide. The order is imposed on the distinct vectors, so edges with the same vector are considered together.

\subsection{Main results}\label{subsec:main-results}

We first state a local construction used in the proof. For a positive integer $M$, let
\[
q_{0,M}=1,\quad q_{t,M}=1+t+\cdots+t^M\quad(t\ge1),
\quad \text{and} \quad b_{k,t}=(k-t)\binom{k}{t}.
\]

\begin{theorem}\label{thm:local-intro}
For every $k\ge3$, $0\le t\le k-1$, and positive integer $M$, there is a $k$-palette $\Palt_{k,t,M}$ such that
\[
\delta_t(\Palt_{k,t,M})\ge (b_{k,t}q_{t,M})^{-(k-t)q_{t,M}},
\]
and every $k$-graph $G$ with $|\partial G|\le M$ that admits $\Palt_{k,t,M}$ is $t$-layered.
\end{theorem}

Our main theorem determines exactly when the coordinate-extension parameter is zero.

\begin{theorem}\label{thm:main-intro}
Let $k\ge3$, $0\le t\le k-1$, and let $F$ be a $k$-graph. Then
\[
\pi_t^{\ext}(F)=0
\quad\Longleftrightarrow\quad
F\text{ is }t\text{-layered}.
\]
Moreover, if $F$ is not $t$-layered, then
\[
\pi_t^{\ext}(F)\ge (b_{k,t}q_{t,|\partial F|})^{-(k-t)q_{t,|\partial F|}}>0.
\]
\end{theorem}

We next compare $t$-layeredness with existing notions.

A linear order on $V(F)$ is a \emph{vanishing order} if every facet occurs in the same coordinate in all ordered edges containing it. Equivalently, there is a map $\iota:\partial F\to[k]$ such that $\iota(p_i(e))=i$ for every edge $e$ and every $i\in[k]$.

A $k$-graph $F$ is \emph{min-layered} if it has a vertex order and a map $\varphi:\partial F\to\N$ such that each associated vector has pairwise distinct entries and, whenever two distinct vectors share a value, that value is the minimum entry of at least one of them. It is \emph{max-layered} if such an order and map can be chosen so that every vector has a unique maximum entry and distinct vectors have distinct maximum values.

\begin{proposition}\label{prop:endpoints}
Let $F$ be a $k$-graph.
\begin{enumerate}[label=\textup{(\roman*)}]
\item $F$ is $0$-layered if and only if it has a vanishing order.
\item $F$ is $1$-layered if and only if it is min-layered.
\item $F$ is $(k-1)$-layered if and only if it is max-layered.
\end{enumerate}
\end{proposition}

The cases \(t=0\) and \((k,t)=(3,1)\) recover known characterizations. 
Indeed, \(\pi^{\mathrm{ext}}_0(F)=\pi_{k-2}(F)\), and the known zero-density characterization \cite{ReiherRodlSchacht2018JLMS} states that \(\pi_{k-2}(F)=0\) if and only if \(F\) has a vanishing order. 
Together with Proposition~\ref{prop:endpoints}(i), this is precisely the case \(t=0\) of Theorem~\ref{thm:main-intro}. 
For 3-graphs, Lamaison~\cite{Lamaison2024} proved that \(\pi^{\mathrm{ext}}_1(F)=0\) if and only if \(F\) is min-layered, which, by Proposition~\ref{prop:endpoints}(ii), is precisely the case \((k,t)=(3,1)\) of Theorem~\ref{thm:main-intro}. 
He also proved that max-layeredness implies \(\pi^{\mathrm{ext}}_2(F)=0\) and asked whether the converse holds.

\begin{question}[Lamaison~\cite{Lamaison2024}]\label{ques}
Is it true that $\pi_2^{\ext}(F)=0$ if and only if  $F$ is max-layered?
\end{question}

The case \((k,t)=(3,2)\) of Theorem \ref{thm:main-intro}, together with Proposition~\ref{prop:endpoints}(iii), answers this question affirmatively. 
Thus Theorem \ref{thm:main-intro} recovers the known cases \(t=0\) and \((k,t)=(3,1)\), settles the remaining case \((k,t)=(3,2)\) raised by Lamaison, 
and extends the characterization to every \(k\ge3\) and \(0\le t\le k-1\).

We also relate max-layeredness to layered hypergraphs.
Ding, Lamaison, Liu, Wang and Yang~\cite{DingLamaisonLiuWangYang2025}
introduced layeredness for $3$-graphs.
We use the extension of this notion to $k$-graphs given by
Yang, Fang and Chen~\cite{YangFangChen2026}.
For a map $f:V(F)\to\N$ and an edge $e$, let $f(e)$ be the multiset $\{f(v):v\in e\}$, where repeated values are kept. A $k$-graph $F$ is \emph{layered} if there is a map $f:V(F)\to\N$ such that every edge has a unique vertex of maximum $f$-value and any two edges with the same maximum value have the same label multiset. Proposition~1 of Yang, Fang and Chen~\cite{YangFangChen2026} states that one may choose $f$ with the following stronger property:
\begin{equation}\label{eq:strong-layered}
|f(e)\cap f(e')|\ge k-1\quad\Longrightarrow\quad f(e)=f(e')
\end{equation}
for all edges $e,e'$, where the intersection is counted with multiplicity.

\begin{theorem}\label{thm:layered}
Every layered $k$-graph is max-layered, and hence is $(k-1)$-layered.
\end{theorem}

We also recall the codegree Tur\'an density. If $H$ is a $k$-graph and $S\in\binom{V(H)}{k-1}$, define
\[
d_H(S)=\bigl|\{v\in V(H)\setminus S:S\cup\{v\}\in E(H)\}\bigr|
\quad \text{and} \quad
\delta_{\co}(H)=\min_{S\in\binom{V(H)}{k-1}}d_H(S).
\]
Let $\operatorname{ex}_{\co}(n,F)$ be the largest value of $\delta_{\co}(H)$ over all $n$-vertex $F$-free $k$-graphs $H$. The \emph{codegree Tur\'an density} of $F$ is
\[
\pi_{\co}(F)=\lim_{n\to\infty}\frac{\operatorname{ex}_{\co}(n,F)}{n}.
\]

Yang, Fang and Chen~\cite{YangFangChen2026} proved that $\pi_{\co}(F)=0$ implies that $F$ is layered. 
Therefore, for every $k$-graph $F$,
\[
\begin{gathered}
\pi_{\co}(F)=0\ \Longrightarrow\ F\text{ is layered}\ \Longrightarrow\ F\text{ is max-layered},\\
F\text{ is max-layered}\ \Longleftrightarrow\ F\text{ is $(k-1)$-layered}\ \Longleftrightarrow\ \pi_{k-1}^{\ext}(F)=0.
\end{gathered}
\]
The first two implications follow from \cite{YangFangChen2026} and
Theorem~\ref{thm:layered}, respectively, while the third and final
equivalences follow from Proposition~\ref{prop:endpoints}(iii) and
Theorem~\ref{thm:main-intro}, respectively.
For $3$-graphs, Yang, Fang and Chen~\cite{YangFangChen2026} also proved that
$\pi_{\co}(F)=0$ if and only if $F$ is layered and the uniform Tur\'an density $\pi_{1}(F)=0$; 
for every linear $3$-graph $F$, $\pi_{\co}(F)=0$ if and only if $F$ is layered.
The converse of Theorem~\ref{thm:layered} fails: the Fano plane is max-layered but not layered, as shown in Proposition~\ref{prop:fano-max-not-layered}.

Our next result gives an increasing sequence of exact values whose first term is the minimum positive value $k^{-k}$ and whose limit is $2/k^k$.
\begin{theorem}\label{thm:sequence}
For every $k\ge3$ and every $r\ge2$, there is a finite $k$-graph $F_{k,r}$ such that
\[
\pi_{k-2}(F_{k,r})=\frac{2(r-1)}{rk^k}.
\]
It follows that $2/k^k$ is an accumulation point of $\Piuk$, the set of $(k-2)$-uniform Tur\'an densities of single finite $k$-graphs. 
\end{theorem}

The sequence starts with $k^{-k},4/(3k^k),3/(2k^k),8/(5k^k),\ldots$. For $k=3$, the first two values are the known values $1/27$ and $4/81$.

We now state condition $\Sp$ of Lin, Wang and Zhou~\cite{LinWangZhou2025}. 
For an ordered edge $e=\{v_1,\ldots,v_k\}$ with $v_1\prec\cdots\prec v_k$, the facet $p_i(e)$ has \emph{type $i$}. 
Fix $1\le i<j\le k$. A $k$-graph $F$ satisfies $\Sp$ for $\{i,j\}$ if there are subgraphs $F^1_{i,j}$ and $F^2_{i,j}$ on the full vertex set $V(F)$ such that
$E(F)=E(F^1_{i,j})\mathbin{\cup}E(F^2_{i,j})$,
and there is a linear order on $V(F)$ that is vanishing for both subgraphs. For every facet $S\in\partial F^1_{i,j}\cap\partial F^2_{i,j}$, let its types in $F^1_{i,j}$ and $F^2_{i,j}$ be $a$ and $b$, respectively. We require
\[
(a,b)\in\{(\ell,\ell):\ell\in[k]\}\cup\{(i,j)\}.
\]
We say that $F$ satisfies $\Sp$ if it satisfies this condition for every pair in $\binom{[k]}2$.

Lin, Wang and Zhou proved that failure of $\Sp$ implies $\pi_{k-2}(F)\ge3/(k+1)^k$ and proposed the following conjecture.

\begin{conjecture}[Lin, Wang and Zhou~\cite{LinWangZhou2025}]\label{conj}
If a $k$-graph $F$ has no vanishing order and satisfies $\Sp$, then $\pi_{k-2}(F)=k^{-k}$.
\end{conjecture}

The following theorem determines the least density of a $k$-graph that fails $\Sp$.
\begin{theorem}\label{thm:threshold}
For every $k\ge3$,
\[
\min\bigl\{\pi_{k-2}(F):F\text{ does not satisfy }\Sp\bigr\}=\frac4{3k^k}.
\]
\end{theorem}

A \emph{matching} in $[k]$ is a family $\mathcal M\subseteq\binom{[k]}2$ whose members are pairwise disjoint. Let $\rho:[k]\to[k]$ be the reversal $\rho(s)=k+1-s$, and extend it to pairs and families by
\[
\rho(\{p,q\})=\{\rho(p),\rho(q)\}\quad \text{and} \quad
\rho(\mathcal M)=\{\rho(h):h\in\mathcal M\}.
\]
We consider non-empty matchings $\mathcal M$ such that $\rho(\mathcal M)=\mathcal M$ and $q-p$ is even for every $\{p,q\}\in\mathcal M$ with $p<q$.

\begin{theorem}\label{thm:matching}
Let $k\ge3$ and let $\mathcal M\subseteq\binom{[k]}2$ be a non-empty matching of size $m$. Suppose that $\rho(\mathcal M)=\mathcal M$ and $q-p$ is even for every $\{p,q\}\in\mathcal M$ with $p<q$. There is a finite $k$-graph $F_{\mathcal M}$ such that
\begin{enumerate}[label=\textup{(\roman*)}]
\item $F_{\mathcal M}$ satisfies $\Sp$ for every pair in $\binom{[k]}2$;
\item $F_{\mathcal M}$ has no vanishing order;
\item $\displaystyle \pi_{k-2}(F_{\mathcal M})=\frac{2^m}{k^k}$.
\end{enumerate}
\end{theorem}

For odd $k$, the matching $\{\{1,k\}\}$ yields density $2/k^k$. 
For even $k\ge4$, the matching $\{\{1,3\},\{k-2,k\}\}$ yields density $4/k^k$. 
Thus Conjecture \ref{conj} is false for every $k\ge3$. Moreover, for odd $k$ and every $1\le s\le(k-1)/2$, the matching
\[
\mathcal M_s=\bigl\{\{1,k\},\{2,k-1\},\ldots,\{s,k+1-s\}\bigr\}
\]
satisfies these conditions. It follows that $2^s/k^k$ is the density of a single finite $k$-graph that satisfies $\Sp$ and has no vanishing order.

\subsection{Organization and proof outline}\label{subsec:outline}

Section~\ref{sec:Preliminaries} recalls palette homomorphisms, reversal, products, symmetrization, the realization theorem, and a palette description of $\Sp$. 
It also establishes the easy direction of Theorem \ref{thm:main-intro} and gives a quotient-digraph characterization of $t$-layered $k$-graphs.
Section~\ref{sec:trees} constructs the universal palette using ordered labelled $t$-ary trees and proves
Theorems~\ref{thm:local-intro} and~\ref{thm:main-intro}. Section~\ref{sec:relation-layered}
derives Proposition \ref{prop:endpoints} and Theorem \ref{thm:layered}. Section~\ref{sec:exact}
develops the symmetrization lemmas needed for Theorem \ref{thm:sequence} and
then uses them to obtain this theorem. Section~\ref{sec:obstructions}
establishes Theorems \ref{thm:threshold} and \ref{thm:matching}. Finally,
Section~\ref{sec:concluding-remarks} records the monotonicity relations
and applies the main theorem to cherry Tur\'an density.

The universal palette uses ordered labelled $t$-ary trees. Each class of facets is assigned a labelled tree. In an admissible tuple, the $t$ tail trees occur, in their coordinate order, as the branches of every head tree. If two occurrences of one head class determine different tuples of classes, we merge the classes in the two tuples coordinate by coordinate and shorten all trees by one level. The number of classes and the tree depth decrease together. When the process stops, every head class determines one ordered tuple. The head classes of each edge form a block. An arc joins the block containing a tail class to the block of head classes of that edge. A directed cycle would force one numerical label to occur at both the root and a lower node on the same root-to-leaf path, which is impossible.

For the exact densities, we use the palette determination and realization theorems of Lin, Sun, Wang and Zhou directly. Symmetrization creates directed cycles in any chosen pair of coordinates. The density upper bounds then follow from a directed walk estimate or from Mantel's theorem.

\section{Preliminaries}\label{sec:Preliminaries}

Let $\Palt=(\C,\T)$ and $\Palt'=(\C',\T')$ be $k$-palettes. A \emph{palette homomorphism} $f:\Palt\to\Palt'$ is a map $f:\C\to\C'$ such that
\[
(c_1,\ldots,c_k)\in\T
\quad\Longrightarrow\quad
(f(c_1),\ldots,f(c_k))\in\T'.
\]
If $F$ admits $\Palt$ and $f:\Palt\to\Palt'$, composing the witnessing coloring with $f$ shows that $F$ admits $\Palt'$. A \emph{palette isomorphism} is a bijective palette homomorphism whose inverse is also a palette homomorphism.

The \emph{reverse palette} is
\[
\rev(\Palt)=\bigl(\C,\{(c_k,\ldots,c_1):(c_1,\ldots,c_k)\in\T\}\bigr).
\]
Reversing a witnessing vertex order shows that $F$ admits $\Palt$ if and only if it admits $\rev(\Palt)$. For palettes $\Palt_\nu=(\C_\nu,\T_\nu)$, $1\le\nu\le s$, their product is the palette with color set $\prod_{\nu=1}^s\C_\nu$ in which a tuple
\[
\bigl((c_{1,1},\ldots,c_{s,1}),\ldots,(c_{1,k},\ldots,c_{s,k})\bigr)
\]
is admissible if and only if $(c_{\nu,1},\ldots,c_{\nu,k})\in\T_\nu$ for every $\nu\in[s]$. 

For $i\in[k]$, the \emph{coordinate support} of $\Palt$ in position $i$ is
\[
C_i(\Palt)=\{c\in\C:c\text{ occurs in coordinate }i\text{ of some tuple in }\T\}.
\]

We next recall symmetrization from~\cite{LinSunWangZhou2026}. Let $\Sym_r$ denote the symmetric group on $[r]$. If $a_1,\ldots,a_r$ are distinct integers, their \emph{relative order} is the unique permutation $\sigma\in\Sym_r$ such that $\sigma(u)<\sigma(v)$ exactly when $a_u<a_v$. For $\tau\in\Sym_k$ and $i\in[k]$, let $\partial_i\tau\in\Sym_{k-1}$ be the relative order of
\[
\tau(1),\ldots,\tau(i-1),\tau(i+1),\ldots,\tau(k).
\]
For every $c\in\C$ and $\sigma\in\Sym_{k-1}$, let $c^\sigma$ be a formal copy of $c$. For $\boldsymbol c=(c_1,\ldots,c_k)\in\T$ and $\tau\in\Sym_k$, define
\[
\tau\cdot\boldsymbol c=
\bigl(c_{\tau(1)}^{\partial_1\tau},\ldots,c_{\tau(k)}^{\partial_k\tau}\bigr).
\]
The \emph{symmetrization} $\Palt^{\Sym}$ has color set $\{c^\sigma:c\in\C,\ \sigma\in\Sym_{k-1}\}$ and admissible tuples $\tau\cdot\boldsymbol c$ for all $\tau\in\Sym_k$ and $\boldsymbol c\in\T$.

We use the following result of Lin, Sun, Wang and Zhou~\cite[Lemma~6.1]{LinSunWangZhou2026}, derived there from their palette equality and multi-palette classification theorems.

\begin{theorem}[{Lin, Sun, Wang and Zhou~\cite[Lemma~6.1]{LinSunWangZhou2026}}]\label{thm:realization}
Let $s\ge1$, let $\alpha\in[0,1]$, and let $\Palt_0,\Palt_1,\ldots,\Palt_s$ be $k$-palettes with $d(\Palt_0)=\alpha$. Assume that:
\begin{enumerate}[label=\textup{(E\arabic*)}]
\item\label{E1} for every $i\in[s]$, there is no homomorphism from
\[
\Palt_i\times\prod_{j\in[s]\setminus\{i\}}\Palt_j^{\Sym}
\]
to $\Palt_0$ or to $\rev(\Palt_0)$;
\item\label{E2} every $k$-palette $\Palt$ with $d(\Palt)>\alpha$ admits a homomorphism from some $\Palt_i$ to $\Palt$ or to $\rev(\Palt)$.
\end{enumerate}
Then there is a finite $k$-graph $F$ such that $\pi_{k-2}(F)=\alpha$, $F$ admits every $\Palt_i$, and $F$ does not admit $\Palt_0$.
\end{theorem}

We also use the characterization $\pi_{k-2}(F)=0$ if and only if $F$ has a vanishing order~\cite{ReiherRodlSchacht2018JLMS,LinWangZhou2025}.

We now describe condition $\Sp$ using a palette. Fix $1\le i<j\le k$. Let $a_1,\ldots,a_k,\omega$ be pairwise distinct colors, and let $\mathcal W_{i,j}$ have exactly the following three admissible tuples:
\[
(a_1,\ldots,a_k),
\quad(a_1,\ldots,a_{i-1},\omega,a_{i+1},\ldots,a_k), 
\quad \text{and} \quad (a_1,\ldots,a_{j-1},\omega,a_{j+1},\ldots,a_k).
\]

\begin{proposition}\label{prop:W-equivalence}
A $k$-graph $F$ satisfies $\Sp$ for $\{i,j\}$ if and only if it admits $\mathcal W_{i,j}$.
\end{proposition}

\begin{proof}
Suppose first that $F$ admits $\mathcal W_{i,j}$. 
Fix a witnessing order and coloring. 
Let $F^1_{i,j}$ contain the edges whose tuple has $\omega$ in coordinate $i$, and let $F^2_{i,j}$ contain all remaining edges. 
Thus the tuples used by $F^2_{i,j}$ are the all $a$-tuple and the tuple with $\omega$ in coordinate $j$.

In $F^1_{i,j}$, the color $\omega$ occurs only in coordinate $i$, and each $a_\ell$ occurs only in coordinate $\ell$. Hence the witnessing order is vanishing for $F^1_{i,j}$. The same argument, with $j$ in place of $i$, applies to $F^2_{i,j}$. Let a facet belong to an edge of each subgraph. Its color is the same in both edges. A color $a_\ell$ means that the facet has type $\ell$ in both subgraphs, while $\omega$ means that it has type $i$ in $F^1_{i,j}$ and type $j$ in $F^2_{i,j}$. Thus the type condition in $\Sp$ holds.

Conversely, suppose that $F^1_{i,j},F^2_{i,j}$ and one vertex order witness $\Sp$ for $\{i,j\}$. The vanishing property makes the type of every facet in each subgraph well defined. 
Color a facet in both shadows by $a_\ell$ if its type is $\ell$ in both subgraphs, and by $\omega$ if its types are $i$ in $F^1_{i,j}$ and $j$ in $F^2_{i,j}$. If a facet lies only in $\partial F^1_{i,j}$, color it by $\omega$ when its type is $i$, and by $a_\ell$ when its type is $\ell\ne i$. If it lies only in $\partial F^2_{i,j}$, use $\omega$ when its type is $j$, and $a_\ell$ when its type is $\ell\ne j$. 

The two edge sets cover $E(F)$, so the two shadows cover $\partial F$. The three cases above are disjoint, and the condition on common facets specifies a unique color in the intersection. Thus the coloring is well defined on $\partial F$.

Let $e\in E(F^1_{i,j})$. For every $\ell\ne i$, the facet $p_\ell(e)$
receives $a_\ell$. The facet $p_i(e)$ receives either $a_i$ or
$\omega$. Hence the tuple of $e$ is either
$(a_1,\ldots,a_k)$
or the tuple obtained from it by replacing $a_i$ with $\omega$.
Similarly, every edge of $F^2_{i,j}$ receives either the all-$a$ tuple or
the tuple with $\omega$ in coordinate $j$. Therefore $F$ admits $\mathcal W_{i,j}$.
\end{proof}

For $p<q$, define the \emph{switch palette} $\mathcal R_{p,q}$ to have two admissible tuples with pairwise distinct entries $\boldsymbol u=(u_1,\ldots,u_k)$ and $\boldsymbol v=(v_1,\ldots,v_k)$. They share exactly one color $z$, where $u_p=v_q=z$; all other entries of the two tuples are pairwise distinct and different from $z$. Define a color map by
\[
z\mapsto\omega,\quad u_i\mapsto a_i\ (i\ne p),\quad \text{and} \quad v_i\mapsto a_i\ (i\ne q).
\]
The tuple $\boldsymbol u$ is mapped to the admissible tuple of $\mathcal W_{p,q}$ with $\omega$ in coordinate $p$, and $\boldsymbol v$ is mapped to the one with $\omega$ in coordinate $q$. Hence
\begin{equation}\label{eq:R-to-W}
\mathcal R_{p,q}\longrightarrow\mathcal W_{p,q}.
\end{equation}

For a digraph $D$, write $V(D)$ for its vertex set and $A(D)$ for its arc set. A directed walk of length $r$ in $D$ is a sequence $(v_0,\ldots,v_r)$ such that $v_{s-1}v_s\in A(D)$ for every $s\in[r]$.

\begin{lemma}[{Lin, Sun, Wang and Zhou~\cite[Lemma~6.4]{LinSunWangZhou2026}}]\label{lem:walk}
Let $n>r\ge1$. If a digraph on $n$ vertices contains no directed walk of length $r$, then it has at most $(r-1)n^2/(2r)$ arcs.
\end{lemma}

The following lemma proves one direction of Theorem \ref{thm:main-intro}. 
Its proof also explains the role of the head coordinates.

\begin{lemma}\label{lem:greedy}
If a $k$-graph $F$ is $t$-layered, then it admits every $k$-palette $\Palt$ with $\delta_t(\Palt)>0$. It follows that $\pi_t^{\ext}(F)=0$.
\end{lemma}

\begin{proof}
Choose the order $\prec$, the map $\varphi$, the distinct vectors $\boldsymbol a^1,\ldots,\boldsymbol a^r$, and the tail sets $I_1,\ldots,I_r$ from Definition~\ref{def:t-layered}. 
We define a color $\gamma(x)\in\C$ for every integer $x$ that occurs in one of the vectors.

Suppose that \(\boldsymbol a^1,\ldots,\boldsymbol a^{s-1}\) have been processed.
For every tail value \(a_i^s\), \(i\in I_s\), retain its color if it has already received one; otherwise, choose any color, using the same color for equal tail values. 
By \ref{D1}, the head values \(a_h^s\), \(h\notin I_s\), are pairwise distinct. 
By \ref{D2}, each of them differs from every current tail value and from every value appearing in an earlier vector. At this point, only values appearing in earlier vectors and the current tail values have received colors.
Hence no current head value has yet received a color.

Since \(\delta_t(\Palt)>0\), the prescribed colors \((\gamma(a_i^s))_{i\in I_s}\) can be extended to an admissible tuple.
Assign to each head value \(a_h^s\), \(h\notin I_s\), the color in coordinate \(h\) of one such completion. 
The pairwise distinction of the head values ensures that these assignments are well defined.

After all vectors have been processed, every value in \(\varphi(\partial F)\) has received a color. 
For \(S\in\partial F\), define \(\chi(S)=\gamma(\varphi(S))\). 
Thus facets with the same \(\varphi\)-value receive the same color. 
For every edge \(e\in E(F)\), its associated vector is one of \(\boldsymbol a^1,\ldots,\boldsymbol a^r\), and the construction ensures that \((\chi(p_1(e)),\ldots,\chi(p_k(e)))\in\T\). 
Therefore \(\prec\) and \(\chi\) witness that \(F\) admits \(\Palt\).
\end{proof}

To prove the converse, we describe $t$-layeredness through a partition of the shadow. Fix a vertex order on a $k$-graph $G$, choose $I(e)\in\binom{[k]}t$ for each edge $e$, and let $\Qcal$ be a partition of $\partial G$. Write $Q(S)$ for the class containing $S$ and
\[
\boldsymbol Q(e)=(Q(p_1(e)),\ldots,Q(p_k(e))).
\]
For this edge, classes in coordinates of $I(e)$ are called tail classes, and those in the remaining coordinates are called head classes. We require
\begin{enumerate}[label=\textup{(Q\arabic*)},ref=\textup{(Q\arabic*)}]
\item\label{Q1} the head classes of each edge are pairwise distinct and differ from all its tail classes;
\item\label{Q2} if $h\notin I(e)$, $h'\notin I(f)$, and $Q(p_h(e))=Q(p_{h'}(f))$, then
\[
I(e)=I(f),\quad  h=h',\quad \text{and} \quad \boldsymbol Q(e)=\boldsymbol Q(f).
\]
\end{enumerate}

For each edge $e$, let $B_e=\{Q(p_h(e)):h\notin I(e)\}$ be its \emph{head block}. By \ref{Q2}, intersecting head blocks are equal: a common class determines the same tuple and the same tail set. Thus the distinct head blocks, together with the singleton sets $\{X\}$ for classes $X$ that are never heads, partition $\Qcal$.

The \emph{dependency digraph} has these blocks as its vertices. For each edge $e$ and each $i\in I(e)$, draw an arc from the block containing $Q(p_i(e))$ to $B_e$. Repeated arcs are counted once. This definition refers to edges and their chosen tail sets; it does not assume that equal tuples already have equal tail sets.

\begin{lemma}\label{lem:quotient}
A $k$-graph $G$ is $t$-layered if and only if there are a vertex order, sets $I(e)\in\binom{[k]}t$ for its edges, and a partition $\Qcal$ of $\partial G$ satisfying \ref{Q1} and \ref{Q2} for which the dependency digraph is acyclic.
\end{lemma}

\begin{proof}
Suppose first that $G$ is $t$-layered. Choose the order, map $\varphi$, vector order, and tail sets in Definition~\ref{def:t-layered}. Partition $\partial G$ by equality of $\varphi$-values. Assign to every edge the tail set of its vector, so edges with equal vectors receive the same tail set.

Condition \ref{Q1} follows from \ref{D1} and \ref{D2}. For \ref{Q2}, suppose that a class is a head for both $e$ and $f$. If their vectors were distinct, its common label would be a head entry in the later vector that had already occurred in the earlier vector, contrary to \ref{D2}. Hence the vectors are equal and have the same tail set. By \ref{D1}, their head entries are distinct, so the common head class occurs in the same coordinate. This proves \ref{Q2}.

Assign to each head block the index of its vector, and assign index zero to each class block that is never a head block. Consider an arc $B_e\to B_f$ between head blocks. Some class in $B_e$ is a tail class of $f$. The vectors of $e$ and $f$ cannot be equal, by \ref{Q1}. The vector of $e$ cannot occur later, because its head label would then have occurred as a tail label in the earlier vector of $f$, contrary to \ref{D2}. Thus the index strictly increases along this arc. An arc from a block of index zero also increases the index. The digraph is therefore acyclic.

Conversely, suppose that the stated order, tail sets, and partition exist. We first show that equal tuples have equal tail sets and equal head blocks. Let $\boldsymbol Q(e)=\boldsymbol Q(f)$. If $B_e\cap B_f\ne\varnothing$, condition \ref{Q2} gives $I(e)=I(f)$ and $B_e=B_f$. Otherwise, a class in $B_e$ occurs as a tail class of $f$, so there is an arc $B_e\to B_f$. Similarly, there is an arc $B_f\to B_e$. This contradicts acyclicity. Thus equal tuples have the same tail set and head block.

Assign distinct positive integers to the classes of $\Qcal$ and label each facet by the integer of its class. Choose a topological order of the dependency digraph, that is, an order of its vertices in which every arc goes forward. Order the distinct tuples $\boldsymbol Q(e)$ by the positions of their head blocks. This is well defined: equal tuples have equal head blocks, and intersecting head blocks determine equal tuples by \ref{Q2}.

The head classes of a tuple are distinct and differ from its tail classes by \ref{Q1}. Suppose that a current head class occurs in an earlier tuple. It cannot be a head class there, since \ref{Q2} would make the tuples equal. If it is a tail class there, an arc goes from the current head block to the earlier head block, contradicting the topological order. Thus no head class occurs in an earlier tuple. The assigned integer vectors satisfy \ref{D1} and \ref{D2}, so $G$ is $t$-layered.
\end{proof}

\section{The universal palette}\label{sec:trees}

We now prove Theorem~\ref{thm:local-intro}. The case $t=0$ does not require trees. Let $\Palt_{k,0,M}$ have color set $[k]$ and the single admissible tuple $(1,2,\ldots,k)$. Then $\delta_0(\Palt_{k,0,M})=k^{-k}$. If a $k$-graph $G$ admits this palette, every edge has the vector $(1,2,\ldots,k)$, so \ref{D1} and \ref{D2} hold with no tail coordinates. Thus $G$ is $0$-layered. Since $b_{k,0}=k$ and $q_{0,M}=1$, the required estimate also holds.

We may therefore suppose that $1\le t\le k-1$. We follow the labelled-tree and partition-merging argument of Yang, Fang and Chen~\cite[Lemmas~2 and~3]{YangFangChen2026}. Here the trees have $t$ ordered branches, and each root records the tail set and the head coordinate it represents. Fix $M\ge1$, and write
\[
q=q_{t,M}=1+t+\cdots+t^M
\quad \text{and} \quad
\Sigma=\bigl\{(I,h):I\in\tbinom{[k]}t,\ h\in[k]\setminus I\bigr\}.
\]
Thus $|\Sigma|=b_{k,t}$.

For $0\le s\le M$, let $T_s$ be the complete rooted $t$-ary tree of depth $s$ with vertex set $\bigcup_{j=0}^s[t]^j$, where $[t]^0=\{\varnothing\}$ and $[t]^j$ is the set of all words of length $j$ over $[t]$. The root is the empty word $\varnothing$, and the $t$ children of a word $w$ of length less than $s$ are $w1,\ldots,wt$. Thus $T_s$ has $1+t+\cdots+t^s$ vertices. For a word $w$, let $|w|$ be its length. If $r\in[t]$, then $rw$ denotes the word obtained by placing $r$ before $w$. A word $u$ is a \emph{proper prefix} of $w$ if $w=uv$ for some non-empty word $v$.

Let $\C_s$ be the collection of labellings $A:T_s\to\Sigma\times[q]$ with the following property. Write $A(w)=(\tau_A(w),\nu_A(w))$, where $\tau_A(w)\in\Sigma$ and $\nu_A(w)\in[q]$, and require
\begin{equation}\label{eq:path-distinct}
\nu_A(u)\ne\nu_A(w)\quad \text{whenever $u$ is a proper prefix of $w$}.
\end{equation}
Thus the numerical labels are distinct along each path from the root; they may repeat in different branches. We identify an element of $\C_s$ with the corresponding labelled tree.

For $A\in\C_s$ and $0\le a\le s$, let $A|_a$ be the restriction of $A$ to $T_a$. In other words, $A|_a$ is obtained by deleting the vertices at depths $a+1,\ldots,s$. In particular, $A|_a\in\C_a$. If $s\ge1$, the subtrees rooted at the children $1,\ldots,t$ all have depth $s-1$. For $r\in[t]$, define $A^r:T_{s-1}\to\Sigma\times[q]$ by $A^r(w)=A(rw)$. Under the identification $w\leftrightarrow rw$, this is the labelling of the subtree rooted at $r$. Since placing the same first letter before two words preserves the ancestor relation, \eqref{eq:path-distinct} gives $A^r\in\C_{s-1}$. Thus $A^1,\ldots,A^t$ are the ordered branches rooted at the children of the root.

Let $\boldsymbol A=(A_1,\ldots,A_k)\in\C_s^k$, where $s\ge1$, and let $I=\{i_1<\cdots<i_t\}\in\binom{[k]}t$. For this tuple and this choice of $I$, consider the following conditions:
\begin{enumerate}[label=\textup{(T\arabic*)}]
\item\label{T1} $\tau_{A_h}(\varnothing)=(I,h)$ for every $h\notin I$;
\item\label{T2} all head trees $A_h$, $h\notin I$, have the same numerical label at the root;
\item\label{T3} $A_h^r=A_{i_r}|_{s-1}$ for every $h\notin I$ and $r\in[t]$.
\end{enumerate}
The coordinates in $I$ are tails and the other coordinates are heads. Thus every head tree records the restrictions of the tail trees in the increasing order of their coordinates. Unlike the multiset relation in~\cite{YangFangChen2026}, these branches retain their order.

\begin{lemma}\label{lem:completion}
For every $I=\{i_1<\cdots<i_t\}\in\binom{[k]}t$ and every choice of tail trees $A_{i_1},\ldots,A_{i_t}\in\C_M$, there are head trees $A_h\in\C_M$, $h\notin I$, such that $(A_1,\ldots,A_k)$ satisfies \ref{T1}--\ref{T3} for $I$.
\end{lemma}

\begin{proof}
Each tree $A_{i_r}|_{M-1}$ has $1+t+\cdots+t^{M-1}$ vertices. Hence the $t$ trees together use at most
\[
t(1+t+\cdots+t^{M-1})=t+t^2+\cdots+t^M=q-1
\]
different numerical labels. Choose $\gamma\in[q]$ that occurs in none of them. For every $h\notin I$, label a new root by $((I,h),\gamma)$ and attach $A_{i_1}|_{M-1},\ldots,A_{i_t}|_{M-1}$ below it as the first through $t$th branches. Call the resulting labelled tree $A_h$.

Every path contained in one branch satisfies \eqref{eq:path-distinct}, because the branch is a restriction of a labelled tree. The numerical label $\gamma$ at the new root occurs in none of the branches. Thus every path from the new root also satisfies \eqref{eq:path-distinct}, and $A_h\in\C_M$. The root pairs satisfy \ref{T1}, all root numerical labels equal $\gamma$ as required by \ref{T2}, and the attached branches satisfy \ref{T3}. Therefore the resulting tuple satisfies \ref{T1}--\ref{T3} for $I$.
\end{proof}

Define
\[
\Palt_{k,t,M}=(\C_M,\T_{k,t,M}),
\]
where a tuple is admissible if it satisfies \ref{T1}--\ref{T3} for at least one $I\in\binom{[k]}t$.

\begin{lemma}\label{lem:palette-degree}
The palette $\Palt_{k,t,M}$ satisfies
\[
\delta_t(\Palt_{k,t,M})\ge (b_{k,t}q)^{-(k-t)q}.
\]
\end{lemma}

\begin{proof}
Let $I\in\binom{[k]}t$ be arbitrary, and prescribe any labelled trees in the coordinates of $I$. By Lemma~\ref{lem:completion}, there is a choice of labelled trees in the other $k-t$ coordinates that completes them to an admissible tuple. Consequently,
\[
\delta_t(\Palt_{k,t,M})\ge\frac{1}{|\C_M|^{k-t}}.
\]
The tree $T_M$ has $q$ vertices, and each vertex has a label in $\Sigma\times[q]$. Hence the total number of such labellings is $(|\Sigma|q)^q=(b_{k,t}q)^q$. Since $\C_M$ consists only of the labellings satisfying \eqref{eq:path-distinct}, we have $|\C_M|\le(b_{k,t}q)^q$. Therefore
\[
\delta_t(\Palt_{k,t,M})\ge(b_{k,t}q)^{-(k-t)q},
\]
as required.
\end{proof}

We record two consequences of \ref{T1}--\ref{T3}. First, if $s\ge2$ and $(A_1,\ldots,A_k)$ satisfies these conditions for $I$, then $(A_1|_{s-1},\ldots,A_k|_{s-1})$ satisfies the same conditions at depth $s-1$. Indeed, restriction does not change the root pairs or numerical labels, and for every $h\notin I$ and $r\in[t]$,
\[
(A_h|_{s-1})^r=(A_h^r)|_{s-2}=A_{i_r}|_{s-2}.
\]
These are precisely the required branches.

Second, the head trees in a tuple satisfying \ref{T1}--\ref{T3} are pairwise distinct and none equals a tail tree. Distinct head trees have different root pairs by \ref{T1}. If $A_h=A_{i_r}$ for some $h\notin I$ and $r\in[t]$, then \ref{T3} yields $A_h^r=A_h|_{s-1}$. Evaluating at the roots gives $\nu_{A_h}(r)=\nu_{A_h}(\varnothing)$. Since $\varnothing$ is a proper prefix of $r$, this contradicts \eqref{eq:path-distinct}.

\begin{lemma}\label{lem:forcing}
Every $k$-graph $G$ with $|\partial G|\le M$ that admits $\Palt_{k,t,M}$ is $t$-layered.
\end{lemma}

\begin{proof}
If $G$ has no edge, the conclusion is immediate. Suppose that $G$ has an edge, and let $\prec$ and $\chi:\partial G\to\C_M$ witness that $G$ admits $\Palt_{k,t,M}$. For each edge $e$, choose a set $I(e)$ such that $(\chi(p_1(e)),\ldots,\chi(p_k(e)))$ satisfies \ref{T1}--\ref{T3} for $I(e)$. If more than one choice is possible, choose any one of them.

We now construct the required partition. Throughout the construction, we maintain a partition $\Qcal$ of $\partial G$, an integer $1\le s\le M$, and a map $\psi:\Qcal\to\C_s$ such that
\begin{equation}\label{eq:merge-invariant}
|\Qcal|\le s
\end{equation}
and, for every edge $e$, the tuple
\begin{equation}\label{eq:quotient-compatible}
\bigl(\psi(Q(p_1(e))),\ldots,\psi(Q(p_k(e)))\bigr)
\end{equation}
satisfies \ref{T1}--\ref{T3} for $I(e)$ at depth $s$. Initially, let $\Qcal$ be the partition into singleton classes, let $s=M$, and define $\psi(\{S\})=\chi(S)$ for every $S\in\partial G$. Since $|\partial G|\le M$, both requirements hold.

Suppose that \ref{Q2} fails for the current partition. Then there are edges $e,f$ and coordinates $h\notin I(e)$, $h'\notin I(f)$ with $Q(p_h(e))=Q(p_{h'}(f))$, but the conclusions of \ref{Q2} do not all hold. The common labelled tree has one root pair, so \ref{T1} gives $I(e)=I(f)=:I$ and $h=h'$. Therefore $\boldsymbol Q(e)\ne\boldsymbol Q(f)$. Every tuple satisfying \ref{T1}--\ref{T3} has at least one head tree and one tail tree, and these labelled trees are different by the second property above. Hence $|\Qcal|\ge2$, and \eqref{eq:merge-invariant} gives $s\ge2$.

Write $I=\{i_1<\cdots<i_t\}$. The common head tree has the same ordered branches in the two tuples. By \ref{T3},
\begin{equation}\label{eq:tail-restriction}
\psi(Q(p_{i_r}(e)))|_{s-1}
=
\psi(Q(p_{i_r}(f)))|_{s-1}
\quad \text{for } r\in[t].
\end{equation}
For every $g\notin I$, the two labelled trees $\psi(Q(p_g(e)))$ and $\psi(Q(p_g(f)))$ have the same root pair $(I,g)$. By \ref{T2}, their numerical root labels both equal the numerical root label of the common head tree. By \ref{T3} and \eqref{eq:tail-restriction}, they also have the same ordered branches. These two labelled trees are therefore equal. Together with \eqref{eq:tail-restriction}, this yields
\begin{equation}\label{eq:all-restrictions}
\psi(Q(p_j(e)))|_{s-1}
=
\psi(Q(p_j(f)))|_{s-1}
\quad \text{for } j\in[k].
\end{equation}

Let $\sim$ be the equivalence relation on the classes of $\Qcal$ generated by $Q(p_j(e))\sim Q(p_j(f))$ for $j\in[k]$. Let $\Qcal'$ be the partition obtained by taking the unions of the $\sim$-equivalence classes. Equivalently, $\Qcal'$ is the finest coarsening of $\Qcal$ in which all these identifications hold.

By \eqref{eq:all-restrictions}, every generating identification joins two old classes whose labelled trees have the same restriction to $T_{s-1}$. By transitivity, all old classes contained in one class of $\Qcal'$ have the same restriction. We may therefore define $\psi'(X')=\psi(X)|_{s-1}$, where $X$ is any old class contained in $X'$. This definition is well defined.

For every edge, restriction preserves \ref{T1}--\ref{T3}, so its new tuple satisfies these conditions for $I(e)$ at depth $s-1$. Moreover, $\boldsymbol Q(e)\ne\boldsymbol Q(f)$, so at least one pair of distinct old classes has been identified. Hence $|\Qcal'|\le|\Qcal|-1\le s-1$. Thus both requirements remain valid after replacing $(\Qcal,s,\psi)$ by $(\Qcal',s-1,\psi')$.

Each merging step strictly decreases the number of classes. Hence the process terminates after finitely many steps. For the terminal partition, \ref{Q2} holds. The second consequence of \ref{T1}--\ref{T3} shows that the head trees of every edge are pairwise distinct and differ from all tail trees. The corresponding classes are therefore different as well, so \ref{Q1} also holds.

It remains to show that the dependency digraph is acyclic. All labelled trees assigned to classes in one head block have the same numerical label at every vertex of their trees: their numerical root labels agree by \ref{T2}, and their ordered branches agree by \ref{T3}. Only the pairs recorded at their roots may differ. The dependency digraph has no loop, since a loop would place a tail class in the head block of the same tuple, contrary to \ref{Q1}.

Suppose for a contradiction that this digraph contains a directed cycle of minimum length,
\[
B_0\longrightarrow B_1\longrightarrow\cdots\longrightarrow B_\ell=B_0.
\]
Then $2\le\ell\le|\Qcal|\le s$. Every block on the cycle has an incoming arc and is therefore a head block. Choose a class $Z_j\in B_j$ for $0\le j<\ell$, and let $Z_\ell=Z_0$. For each $0\le j<\ell$, the arc $B_j\to B_{j+1}$ comes from an edge whose head block is $B_{j+1}$ and whose $r_j$th tail class, say $X_j$, belongs to $B_j$, for some $r_j\in[t]$. Condition~\ref{T3} and the agreement of numerical labels within $B_j$ imply
\begin{equation}\label{eq:cycle-transfer}
\nu_{\psi(Z_{j+1})}(r_jw)
=
\nu_{\psi(X_j)}(w)
=
\nu_{\psi(Z_j)}(w)
\quad  \text{for } w\in T_{s-1}.
\end{equation}

Let $w_0=\varnothing$ and $\alpha=\nu_{\psi(Z_0)}(\varnothing)$. We prove inductively that, for every $0\le j\le\ell$, there is a word $w_j$ of length $j$ over $[t]$ such that $\nu_{\psi(Z_j)}(w_j)=\alpha$. This is true for $j=0$. Suppose that it holds for some $j<\ell$. Since $j\le\ell-1\le s-1$, the word $w_j$ belongs to $T_{s-1}$. Define $w_{j+1}=r_jw_j$. Using \eqref{eq:cycle-transfer}, we obtain
\[
\nu_{\psi(Z_{j+1})}(w_{j+1})
=\nu_{\psi(Z_j)}(w_j)=\alpha.
\]
This proves the assertion.

Finally, since $Z_\ell=Z_0$, the numerical label $\alpha$ occurs in $\psi(Z_0)$ both at the root $\varnothing$ and at the vertex $w_\ell$ of depth $\ell\ge1$. Since $\varnothing$ is a proper prefix of $w_\ell$, this contradicts \eqref{eq:path-distinct}. Therefore the dependency digraph is acyclic. The partition $\Qcal$ satisfies all conditions in Lemma~\ref{lem:quotient}, which implies that $G$ is $t$-layered.
\end{proof}

\begin{proof}[\bfseries{Proof of Theorem~\ref{thm:local-intro}}]
For $t=0$, the one-tuple palette defined at the start of this section has the required properties. Suppose that $1\le t\le k-1$, and use $\Palt_{k,t,M}$. By Lemma~\ref{lem:palette-degree},
\[
\delta_t(\Palt_{k,t,M})\ge(b_{k,t}q_{t,M})^{-(k-t)q_{t,M}}.
\]
By Lemma~\ref{lem:forcing}, every $k$-graph $G$ with $|\partial G|\le M$ that admits this palette is $t$-layered. This completes the proof.
\end{proof}

\begin{proof}[\bfseries{Proof of Theorem~\ref{thm:main-intro}}]
Suppose first that $F$ is $t$-layered. By Lemma~\ref{lem:greedy}, $F$ admits every palette with positive $t$th coordinate-extension degree. Hence every palette not admitted by $F$ has $t$th coordinate-extension degree zero, and $\pi_t^{\ext}(F)=0$.

Conversely, let $F$ be a $k$-graph that is not $t$-layered, and let $M=|\partial F|$. Since an edgeless $k$-graph is $t$-layered, $M\ge1$. By Theorem~\ref{thm:local-intro}, $F$ does not admit $\Palt_{k,t,M}$; otherwise, it would be $t$-layered. Therefore, by the definition of $\pi_t^{\ext}(F)$,
\[
\pi_t^{\ext}(F)
\ge\delta_t(\Palt_{k,t,M})
\ge(b_{k,t}q_{t,M})^{-(k-t)q_{t,M}}>0.
\]
This proves both the equivalence and the stated lower bound.
\end{proof}

\section{Relation between \texorpdfstring{$t$}{t}-layeredness and known concepts}\label{sec:relation-layered}

We restate Proposition~\ref{prop:endpoints} before proving it.
\begin{proposition}
Let $F$ be a $k$-graph.
\begin{enumerate}[label=\textup{(\roman*)}]
\item $F$ is $0$-layered if and only if it has a $(k-1)$-vanishing order.
\item $F$ is $1$-layered if and only if it is min-layered.
\item $F$ is $(k-1)$-layered if and only if it is max-layered.
\end{enumerate}
\end{proposition}

\begin{proof}[Proof]
For (i), suppose first that $F$ is $0$-layered. Since there are no tail coordinates, \ref{D1} and \ref{D2} show that the entries in each vector are pairwise distinct and that two distinct vectors have disjoint sets of entries. If one facet belongs to two edges, its label occurs in both vectors. The two vectors must therefore be equal, and the facet has the same coordinate in both edges. Thus the witnessing vertex order is vanishing. Conversely, under a vanishing order, label each facet by its well-defined coordinate. Every edge then has the vector $(1,2,\ldots,k)$, which satisfies \ref{D1} and \ref{D2} with $t=0$.

For (ii), suppose first that $F$ is min-layered, and let $\mathcal A$ be the corresponding set of distinct vectors. Order its vectors as $\boldsymbol a^1,\ldots,\boldsymbol a^r$ so that their minimum entries are non-decreasing. Every vector has a unique minimum. For each $s$, let $I_s$ contain the coordinate of this minimum. The entries outside $I_s$ are pairwise distinct and differ from the tail entry. It is enough to show that no head entry occurs in an earlier vector. Suppose for a contradiction that a head value $a_h^s$, with $h\notin I_s$, occurs in $\boldsymbol a^u$ for some $u<s$. Since $a_h^s$ is not the minimum of $\boldsymbol a^s$, the definition of min-layeredness implies that it is the minimum of $\boldsymbol a^u$. Hence
\[
a_h^s=\min_i a_i^u\le\min_i a_i^s,
\]
which contradicts $a_h^s>\min_i a_i^s$. Thus this vector order satisfies \ref{D1} and \ref{D2} with $t=1$.

Conversely, suppose that $F$ is $1$-layered, and choose a vector order $\boldsymbol a^1,\ldots,\boldsymbol a^r$ satisfying \ref{D1} and \ref{D2} with $t=1$. Construct a digraph whose vertices are the labels occurring in these vectors. For each vector, draw an arc from its tail value to every head value. A label can be a head in at most one vector, by \ref{D2}. If a label $x$ is a head in vector $s(x)$ and $x\to y$ is an arc, then $x$ occurs as the tail in vector $s(y)$, and \ref{D2} gives $s(x)<s(y)$. Every vertex on a directed cycle would be a head, so these strict inequalities show that the digraph is acyclic. Relabel its vertices by distinct positive integers in a topological order. In every vector, the tail label is then smaller than all head labels and is its unique minimum. If two distinct vectors share a label, that label is a head in at most one of them and is therefore the tail, hence the minimum, in the other. The relabelled vectors therefore show that $F$ is min-layered.

For (iii), suppose first that $F$ is max-layered. Order the distinct vectors by increasing maximum value, and in each vector take the coordinate of its unique maximum as the head coordinate. The head value is different from every tail value. It has not occurred in an earlier vector, because every entry in an earlier vector is at most its smaller maximum. This vector order satisfies \ref{D1} and \ref{D2} with $t=k-1$.

Conversely, suppose that the distinct vectors satisfy \ref{D1} and \ref{D2} with $t=k-1$. Each vector has one head. Draw an arc from every tail value to the head value of its vector. A head label occurs as a head in exactly one vector. If a tail label is the head of another vector, that other vector occurs earlier by \ref{D2}. Hence, along every arc whose tail is also a head, the index of the head vector strictly increases. As above, the digraph is acyclic. Relabel all values by distinct positive integers in a topological order. The head is then the unique maximum of its vector, and distinct vectors have distinct head values by \ref{D2}. Thus the relabelled vectors satisfy the definition of max-layeredness. This completes the proof.
\end{proof}

\begin{proof}[\bfseries{Proof of Theorem \ref{thm:layered}}]
Choose a map $f:V(F)\to\N$ that satisfies the definition of layeredness and \eqref{eq:strong-layered}. Order the vertices by non-decreasing $f$-value, breaking ties arbitrarily. Let $\mathcal B=\{p_k(e):e\in E(F)\}$; these are the facets obtained by deleting the unique maximum vertex of an edge.

We claim that each edge $e$ contains exactly one member of $\mathcal B$, namely $p_k(e)$. Suppose that $S\in\mathcal B$ is contained in $e$. By definition, there is an edge $S\cup\{z\}$ with $f(z)>f(x)$ for every $x\in S$. Write $e=S\cup\{w\}$. The two edges share the $k-1$ labels on $S$, counted with multiplicity, so \eqref{eq:strong-layered} makes their full label multisets equal. Removing the labels on $S$ gives $f(w)=f(z)$. Hence $w$ is the unique maximum vertex of $e$, and $S=p_k(e)$. This proves the claim.

Assign distinct integers greater than $1$ to the members of $\mathcal B$ and label every other facet by $1$. Each edge then has a vector of the form $(1,\ldots,1,L)$ with $L>1$. Its maximum is unique, and distinct vectors have distinct maximum values. Thus $F$ is max-layered. Proposition~\ref{prop:endpoints}(iii) implies that it is also $(k-1)$-layered.
\end{proof}

The Fano plane shows that the converse fails.

\begin{proposition}\label{prop:fano-max-not-layered}
The Fano plane $\mathbb{F}$ is max-layered but not layered.
\end{proposition}

\begin{proof}
First we show that $\mathbb{F}$ is max-layered. Since $\mathbb{F}$ is a linear $3$-graph, two distinct edges share no pair. For every edge $e\in E(\mathbb{F})$, choose a pair $S_e\in\binom{e}{2}$. The pairs $S_e$ are pairwise distinct, and each $S_e$ is contained in no edge other than $e$. Choose pairwise distinct integers $L_e>1$ for $e\in E(\mathbb{F})$, and define $\varphi:\partial\mathbb{F}\to\mathbb{N}$ by
\[
\varphi(S)=
\begin{cases}
L_e,&\text{if }S=S_e\text{ for some }e\in E(\mathbb{F}),\\
1,&\text{otherwise}.
\end{cases}
\]
Fix any linear order $\prec$ on $V(\mathbb{F})$. For an edge $e=\{v_1,v_2,v_3\}$ with $v_1\prec v_2\prec v_3$, exactly one of its three pairs is $S_e$, and no other chosen pair $S_f$ with $f\ne e$ is contained in $e$. Hence the associated vector of $e$ is a coordinate permutation of $(1,1,L_e)$. It has the unique maximum value $L_e$, and distinct edges have distinct maximum values. Thus $\mathbb{F}$ is max-layered.

We next show that $\mathbb{F}$ is not layered. Suppose for a contradiction that a map $f:V(\mathbb{F})\to\mathbb{N}$ witnesses that $\mathbb{F}$ is layered. Let $M=\max\{f(v):v\in V(\mathbb{F})\}$. Since every two vertices of $\mathbb{F}$ lie together in an edge, exactly one vertex has value $M$; otherwise, an edge containing two vertices of value $M$ would have no unique vertex of maximum value.

After relabelling the vertices, we may suppose that $f(1)=M$ and
\[
E(\mathbb{F})=\{123,145,167,246,257,347,356\},
\]
where $ijk$ denotes the edge $\{i,j,k\}$. The three edges $123$, $145$, and $167$ have the same maximum value $M$. Since $f$ witnesses that $\mathbb{F}$ is layered, these three edges have the same label multiset. Removing the common entry $M$, we conclude that the three multisets $\{f(2),f(3)\}$, $\{f(4),f(5)\}$, and $\{f(6),f(7)\}$ are equal. Hence there are integers $a,b<M$ such that each of the pairs $\{2,3\}$, $\{4,5\}$, and $\{6,7\}$ receives the label multiset $\{a,b\}$.

If $a=b$, then every vertex in $\{2,\ldots,7\}$ has label $a$. In particular, the edge $246$ has no unique vertex of maximum value, a contradiction. We may therefore suppose that $a<b$.

Let $B=\{v\in\{2,\ldots,7\}:f(v)=b\}$. Each of the pairs $\{2,3\}$, $\{4,5\}$, and $\{6,7\}$ contains exactly one vertex of $B$, so $|B|=3$. Consider the four edges not containing $1$, namely
\[
\mathcal E_0=\{246,257,347,356\}.
\]
Every vertex of an edge in $\mathcal E_0$ has label either $a$ or $b$. Since $b>a$, such an edge has a unique vertex of maximum value if and only if it contains exactly one vertex of $B$. Thus $\sum_{e\in\mathcal E_0}|e\cap B|=4$. On the other hand, every vertex in $\{2,\ldots,7\}$ belongs to exactly two edges of $\mathcal E_0$, and therefore $\sum_{e\in\mathcal E_0}|e\cap B|=2|B|=6$, a contradiction.

Hence $\mathbb{F}$ is not layered. This completes the proof.
\end{proof}

\section{Symmetrization and exact small densities}\label{sec:exact}

For a $k$-palette $\Palt=(\C,\T)$ and $p<q$, its \emph{$(p,q)$-coordinate digraph} has vertex set $\C$ and an arc $xy$ whenever some tuple in $\T$ has $x$ in coordinate $p$ and $y$ in coordinate $q$; a loop is allowed. A \emph{digraph homomorphism} from $D$ to $D'$ is a map $g:V(D)\to V(D')$ such that $g(x)g(y)$ is an arc of $D'$ whenever $xy$ is an arc of $D$.

We first record an elementary consequence of the definition of symmetrization.

\begin{lemma}\label{lem:unique-placement}
Let $p,q\in[k]$ and $\sigma\in\Sym_{k-1}$. There is exactly one permutation $\tau\in\Sym_k$ such that
\[
\tau(p)=q
\quad\text{and}\quad
\partial_p\tau=\sigma.
\]
\end{lemma}

\begin{proof}
For convenience, write $[k]\setminus\{p\}=\{j_1<\cdots<j_{k-1}\}$ and $[k]\setminus\{q\}=\{b_1<\cdots<b_{k-1}\}$. Define $\tau\in\Sym_k$ by
\[
\tau(p)=q 
\quad \text{and} \quad
\tau(j_u)=b_{\sigma(u)} \text{ for } 1\le u\le k-1.
\]
Since $\sigma$ is a permutation of $[k-1]$, the values $b_{\sigma(1)},\ldots,b_{\sigma(k-1)}$ are precisely the elements of $[k]\setminus\{q\}$. Hence $\tau$ is a permutation of $[k]$. After deleting the entry $\tau(p)=q$, the remaining sequence is $(b_{\sigma(1)},\ldots,b_{\sigma(k-1)})$, whose relative order is $\sigma$. Thus $\partial_p\tau=\sigma$.

To prove uniqueness, let $\tau'\in\Sym_k$ satisfy $\tau'(p)=q$ and $\partial_p\tau'=\sigma$. The entries outside position $p$ are $b_1,\ldots,b_{k-1}$, and their relative order is $\sigma$. Consequently, $\tau'(j_u)=b_{\sigma(u)}$ for every $u\in[k-1]$, so $\tau'=\tau$.
\end{proof}

The following lemma extends the cyclic shift of the endpoint coordinates to any interval of coordinates.

\begin{lemma}[Coordinate-cycle lemma]\label{lem:cycle}
Let $\Palt$ be a $k$-palette containing an admissible tuple with pairwise distinct entries, and let $1\le p<q\le k$. In $\Palt^{\Sym}$ there are distinct colors $u_0,\ldots,u_{q-p}$ and admissible tuples whose colors in coordinates $p,q$ are
\[
(u_0,u_1),(u_1,u_2),\ldots,(u_{q-p},u_0).
\]
Thus the coordinate digraph in positions $p,q$ contains a directed cycle of length $q-p+1$.
\end{lemma}

\begin{proof}
Fix an admissible tuple with pairwise distinct entries $\boldsymbol c=(c_1,\ldots,c_k)$ of $\Palt$. For $\tau\in\Sym_k$, define $\Phi(\tau)$ by moving the entry in position $q$ to position $p$ and moving the entries in positions $p,\ldots,q-1$ one place to the right. Thus, in one-line form,
\[
\Phi(\tau)=\bigl(\tau(1),\ldots,\tau(p-1),\tau(q),\tau(p),\ldots,
\tau(q-1),\tau(q+1),\ldots,\tau(k)\bigr).
\]
Deleting the entry in position $q$ from $\tau$ and deleting the entry in position $p$ from $\Phi(\tau)$ produce the same ordered $(k-1)$-tuple. Hence
\begin{equation}\label{eq:partial-shift}
\Phi(\tau)(p)=\tau(q)
\quad \text{and} \quad
\partial_p\Phi(\tau)=\partial_q\tau.
\end{equation}
The color in coordinate $q$ of $\tau\cdot\boldsymbol c$ is therefore
\[
c_{\tau(q)}^{\partial_q\tau}
=c_{\Phi(\tau)(p)}^{\partial_p\Phi(\tau)},
\]
which is exactly the color in coordinate $p$ of $\Phi(\tau)\cdot\boldsymbol c$.

Choose any $\tau_0\in\Sym_k$ and let $\tau_s=\Phi^s(\tau_0)$ for $s\ge0$. Write $d=q-p+1$. The operation $\Phi$ cyclically shifts the $d$ entries in positions $p,\ldots,q$, so $\tau_d=\tau_0$. For $0\le s<d$, let $u_s$ be the color in coordinate $p$ of $\tau_s\cdot\boldsymbol c$. By \eqref{eq:partial-shift}, the colors in coordinates $p,q$ of this tuple are $(u_s,u_{s+1})$, with indices taken modulo $d$.

The underlying colors of $u_0,\ldots,u_{d-1}$ are precisely the $d$ colors $c_{\tau_0(p)},\ldots,c_{\tau_0(q)}$, in some order. Since the entries of $\boldsymbol c$ are pairwise distinct, these colors are pairwise distinct. Hence the formal colors $u_0,\ldots,u_{d-1}$ are also pairwise distinct. The chosen tuples therefore witness a directed cycle of length $d$ in the required coordinate digraph.
\end{proof}

We use two consequences. The first uses a repeated color to obtain a switch palette.

\begin{lemma}\label{lem:repeated-to-sym}
Let $\mathcal A$ be a $k$-palette. Suppose that there are admissible tuples $\boldsymbol u,\boldsymbol v$ and coordinates $a,b\in[k]$ such that $u_a=v_b=x$ and $(\boldsymbol u,a)\ne(\boldsymbol v,b)$. Then, for every $p<q$,
\[
\mathcal R_{p,q}\longrightarrow\mathcal A^{\Sym}.
\]
\end{lemma}

\begin{proof}
Fix $\sigma\in\Sym_{k-1}$. By Lemma~\ref{lem:unique-placement}, there are permutations $\tau,\tau'\in\Sym_k$ satisfying
\[
\tau(p)=a,\quad\partial_p\tau=\sigma,
\quad
\tau'(q)=b,\quad \text{and} \quad \partial_q\tau'=\sigma.
\]
The admissible tuples $\tau\cdot\boldsymbol u$ and $\tau'\cdot\boldsymbol v$ of $\mathcal A^{\Sym}$ therefore have the same formal color $x^\sigma$ in coordinates $p$ and $q$, respectively. Map the special color of $\mathcal R_{p,q}$ to $x^\sigma$. Every other color of $\mathcal R_{p,q}$ occurs exactly once, so map it to the color in the corresponding coordinate of the appropriate chosen tuple. These assignments are compatible, and each admissible tuple of $\mathcal R_{p,q}$ is mapped to one of the two chosen tuples. Thus they define the required homomorphism.
\end{proof}

For a digraph $D=(V,A)$ and coordinates $p<q$, define the palette $\mathcal Q(D;p,q)$ as follows. Its color set contains $V$ and, for every arc $xy\in A$ and every $\ell\in[k]\setminus\{p,q\}$, a new color $d_{xy,\ell}$. All these new colors are pairwise distinct and are disjoint from $V$. For every arc $xy$, the palette has one admissible tuple with $x$ in coordinate $p$, $y$ in coordinate $q$, and $d_{xy,\ell}$ in coordinate $\ell$.

\begin{lemma}\label{lem:graph-to-sym}
Let $\mathcal A$ be a $k$-palette containing an admissible tuple with pairwise distinct entries, let $p<q$, and let $d=q-p+1$. If a digraph $D$ admits a homomorphism to the directed cycle of length $d$, then
\[
\mathcal Q(D;p,q)\longrightarrow\mathcal A^{\Sym}.
\]
\end{lemma}

\begin{proof}
By Lemma~\ref{lem:cycle}, there are distinct colors $u_0,\ldots,u_{d-1}$ and admissible tuples $\boldsymbol b^0,\ldots,\boldsymbol b^{d-1}$ of $\mathcal A^{\Sym}$ such that $\boldsymbol b^s$ has $u_s$ in coordinate $p$ and $u_{s+1}$ in coordinate $q$, with indices modulo $d$. Identify the vertices of the directed $d$-cycle with $0,\ldots,d-1$ in cyclic order, and let $g$ be a homomorphism from $D$ to this cycle.

Map every vertex color $x$ of $\mathcal Q(D;p,q)$ to $u_{g(x)}$. For every arc $xy$ of $D$, we have $g(y)=g(x)+1$ modulo $d$. Map the additional color $d_{xy,\ell}$ to the entry in coordinate $\ell$ of $\boldsymbol b^{g(x)}$ for every $\ell\notin\{p,q\}$. All additional colors are distinct and disjoint from the vertex colors, so the map is well defined. The tuple corresponding to $xy$ is mapped exactly to $\boldsymbol b^{g(x)}$. Thus the map is a palette homomorphism.
\end{proof}

We now prove Theorem~\ref{thm:sequence}. Fix $k\ge3$ and $r\ge2$. 
We first define the palette that determines the lower bound. Take pairwise disjoint color classes
\[
A_1,\ldots,A_r,D_2,\ldots,D_{k-1}
\]
with $|A_s|=2$ for $s\in[r]$ and $|D_i|=r$ for $2\le i\le k-1$. The target palette $\mathcal P_{k,r}$ has the union of these classes as its color set. Clearly, it has $kr$ colors. A tuple $(x_1,\ldots,x_k)$ is admissible if and only if there are $1\le a<b\le r$ such that
\[
x_1\in A_a,\quad x_k\in A_b,
\quad \text{and} \quad x_i\in D_i \quad\text{for all } 2\le i\le k-1.
\]
For every pair $a<b$, there are four choices for $(x_1,x_k)$ and $r^{k-2}$ choices for the middle coordinates. Since the sets $A_1,\ldots,A_r$ are pairwise disjoint, choices from different pairs do not overlap. Therefore, there are $4\binom r2r^{k-2}$ admissible tuples, and
\begin{equation}\label{eq:target-density}
d(\mathcal P_{k,r})=
\frac{4\binom r2r^{k-2}}{(kr)^k}
=\frac{2(r-1)}{rk^k}.
\end{equation}
The reverse palette is isomorphic to $\mathcal P_{k,r}$: for every $s$ and $i$, choose a bijection from $A_s$ to $A_{r+1-s}$ and a bijection from $D_i$ to $D_{k+1-i}$.

We next define the auxiliary palettes. The palette $\mathcal Q_r$ has colors $c_0,\ldots,c_r$ and, for every $s\in[r]$ and $2\le\ell\le k-1$, a new color $d_{s,\ell}$. All these colors are pairwise distinct. Its admissible tuples are
\[
(c_{s-1},d_{s,2},\ldots,d_{s,k-1},c_s) \quad\text{for } s\in[r].
\]
For every pair $p<q$ other than $\{1,k\}$, we also use the switch palette $\mathcal R_{p,q}$. These are the auxiliary palettes used in Theorem~\ref{thm:realization}.

We first verify the non-homomorphism condition and the homomorphisms needed for \ref{E1} of Theorem~\ref{thm:realization}.

\begin{lemma}\label{lem:no-target-sequence}
There is no homomorphism from any auxiliary palette to $\mathcal P_{k,r}$ or to $\rev(\mathcal P_{k,r})$.
\end{lemma}

\begin{proof}
Fix $p<q$ with $\{p,q\}\ne\{1,k\}$. The special color of $\mathcal R_{p,q}$ occurs in coordinate $p$ of its first tuple and in coordinate $q$ of its second tuple. In every admissible tuple of $\mathcal P_{k,r}$, an endpoint color belongs to $A_1\cup\cdots\cup A_r$, while the color in each middle coordinate $i$ belongs to $D_i$. These sets are pairwise disjoint. Thus, unless $\{p,q\}=\{1,k\}$, the supports of coordinates $p$ and $q$ are disjoint. A homomorphism would force the image of the special color to belong to both supports, which is impossible. The same argument applies to the reverse palette, since reversal permutes the middle supports and exchanges the two endpoint coordinates.

Suppose for a contradiction that $f:\mathcal Q_r\to\mathcal P_{k,r}$ is a homomorphism. Every $f(c_s)$ belongs to exactly one of $A_1,\ldots,A_r$; write $f(c_s)\in A_{b_s}$. The image of the $s$th tuple of $\mathcal Q_r$ is admissible only if $b_{s-1}<b_s$. Hence
\(
b_0<b_1<\cdots<b_r,
\)
which is impossible because every $b_s$ lies in $[r]$. A homomorphism to $\rev(\mathcal P_{k,r})$ would instead imply $b_0>b_1>\cdots>b_r$, which is also impossible. This proves the lemma.
\end{proof}

\begin{lemma}\label{lem:pairwise-transfer-sequence}
If $\mathcal A$ and $\mathcal B$ are distinct auxiliary palettes, then
\[
\mathcal A\longrightarrow\mathcal B^{\Sym}.
\]
\end{lemma}

\begin{proof}
If $\mathcal A=\mathcal R_{p,q}$, then every possible palette $\mathcal B$ contains a color in two tuple-coordinate positions: the shared color when $\mathcal B$ is a switch palette, and $c_1$ when $\mathcal B=\mathcal Q_r$, since $r\ge2$. Lemma~\ref{lem:repeated-to-sym} gives $\mathcal A\to\mathcal B^{\Sym}$.

If $\mathcal A=\mathcal Q_r$, then $\mathcal B$ is a switch palette and contains an admissible tuple with pairwise distinct entries. The directed path $c_0\to\cdots\to c_r$ maps to the directed $k$-cycle by $c_s\mapsto s\pmod k$. Since $\mathcal Q_r$ is precisely the palette associated with this path in coordinates $1,k$, Lemma~\ref{lem:graph-to-sym} gives $\mathcal Q_r\to\mathcal B^{\Sym}$.
\end{proof}

We next prove the density estimate needed for \ref{E2}.

\begin{lemma}\label{lem:upper-sequence}
Let $\Palt=(\C,\T)$ be a $k$-palette. Suppose that no auxiliary palette admits a homomorphism to $\Palt$. Then
\[
d(\Palt)\le\frac{2(r-1)}{rk^k}.
\]
\end{lemma}

\begin{proof}
Let $n=|\C|$, and write $C_i=C_i(\Palt)$ for every $i\in[k]$. We first show that the coordinate supports are separated except possibly at the endpoints. Fix $p<q$ with $\{p,q\}\ne\{1,k\}$. If $x\in C_p\cap C_q$, choose an admissible tuple having $x$ in coordinate $p$ and another having $x$ in coordinate $q$. Map the special color of $\mathcal R_{p,q}$ to $x$, and map every other color to the corresponding coordinate of the chosen tuples. Since all these other colors are pairwise distinct, this is a well-defined homomorphism $\mathcal R_{p,q}\to\Palt$, a contradiction. Therefore
\begin{equation}\label{eq:support-separation-sequence}
C_p\cap C_q=\varnothing
\quad\text{for all }\{p,q\}\ne\{1,k\}.
\end{equation}

Let $U=C_1\cup C_k$. Define a digraph $D$ on $U$, allowing loops at this point, by placing an arc $x\to y$ whenever there are colors $y_2,\ldots,y_{k-1}$ such that $(x,y_2,\ldots,y_{k-1},y)\in\T$. We claim that $D$ contains no directed walk of length $r$. Suppose for a contradiction that $v_0\to v_1\to\cdots\to v_r$ is such a walk. For every $s\in[r]$, choose an admissible tuple witnessing the arc $v_{s-1}\to v_s$. Map $c_s$ to $v_s$ for $0\le s\le r$, and map the additional colors of the $s$th tuple of $\mathcal Q_r$ to the middle coordinates of its chosen witness. Since the additional colors are all distinct, this is a well-defined homomorphism $\mathcal Q_r\to\Palt$, a contradiction.

If $|U|>r$, Lemma~\ref{lem:walk} gives $|A(D)|\le(r-1)|U|^2/(2r)$. If $|U|\le r$, then $D$ has no directed cycle; otherwise, repeatedly traversing the cycle would produce a directed walk of length $r$. In particular, $D$ has no loop, and its acyclicity gives $|A(D)|\le\binom{|U|}{2}$. Since $|U|\le r$, we have $\binom{|U|}{2}\le(r-1)|U|^2/(2r)$. Thus in all cases
\begin{equation}\label{eq:walk-bound-sequence}
|A(D)|\le\frac{r-1}{2r}|U|^2.
\end{equation}

For every fixed arc of $D$, there are at most $\prod_{i=2}^{k-1}|C_i|$ ordered choices for the middle coordinates of an admissible tuple. By \eqref{eq:support-separation-sequence}, the sets $U,C_2,\ldots,C_{k-1}$ are pairwise disjoint. Hence $\sum_{i=2}^{k-1}|C_i|\le n-|U|$, and the arithmetic--geometric mean inequality yields
\[
\prod_{i=2}^{k-1}|C_i|
\le\left(\frac{n-|U|}{k-2}\right)^{k-2}.
\]
It follows from \eqref{eq:walk-bound-sequence} that
\[
|\T|\le|A(D)|\prod_{i=2}^{k-1}|C_i|
\le\frac{r-1}{2r}|U|^2
\left(\frac{n-|U|}{k-2}\right)^{k-2}.
\]

Let $x=|U|/n$. Dividing by $n^k$, we obtain
\[
d(\Palt)\le\frac{r-1}{2r(k-2)^{k-2}}x^2(1-x)^{k-2}.
\]
For $0<x<1$, differentiation gives
\[
\frac{d}{dx}\bigl(x^2(1-x)^{k-2}\bigr)
=x(1-x)^{k-3}(2-kx).
\]
The function is increasing on $[0,2/k]$ and decreasing on $[2/k,1]$. Therefore,
\[
\max_{0\le x\le1}x^2(1-x)^{k-2}
=\left(\frac2k\right)^2\left(\frac{k-2}{k}\right)^{k-2}
=\frac{4(k-2)^{k-2}}{k^k}.
\]
Combining the last two bounds, we conclude that
\[
d(\Palt)\le\frac{r-1}{2r(k-2)^{k-2}}
\cdot\frac{4(k-2)^{k-2}}{k^k}
=\frac{2(r-1)}{rk^k}.
\]
This proves the lemma.
\end{proof}

\begin{proof}[\bfseries{Proof of Theorem~\ref{thm:sequence}}]
Apply Theorem~\ref{thm:realization} to the target palette $\mathcal P_{k,r}$ and to the auxiliary palettes $\mathcal Q_r$ and $\mathcal R_{p,q}$ with $\{p,q\}\ne\{1,k\}$. Equation \eqref{eq:target-density} states that $d(\mathcal P_{k,r})=2(r-1)/(rk^k)$.

We first verify \ref{E1}. Fix a auxiliary palette $\mathcal A$. For every other auxiliary palette $\mathcal B$, choose a homomorphism $f_{\mathcal B}:\mathcal A\to\mathcal B^{\Sym}$ as in Lemma~\ref{lem:pairwise-transfer-sequence}. Together with the identity homomorphism on $\mathcal A$, these maps define
\[
x\longmapsto\bigl(x,(f_{\mathcal B}(x))_{\mathcal B\ne\mathcal A}\bigr).
\]
By the definition of a palette product, this is a homomorphism from $\mathcal A$ to its mixed product with the symmetrizations of all other auxiliary palettes. If the mixed product admitted a homomorphism to $\mathcal P_{k,r}$ or to $\rev(\mathcal P_{k,r})$, composition would contradict Lemma~\ref{lem:no-target-sequence}. Thus \ref{E1} holds.

For \ref{E2}, let $\Palt$ be a $k$-palette with $d(\Palt)>2(r-1)/(rk^k)$. If no auxiliary palette admitted a homomorphism to $\Palt$, Lemma~\ref{lem:upper-sequence} would give the opposite inequality. Therefore some auxiliary palette admits a homomorphism to $\Palt$, which is sufficient for \ref{E2}.

The realization theorem now provides a finite $k$-graph $F_{k,r}$ that admits every auxiliary palette, does not admit $\mathcal P_{k,r}$, and satisfies
\[
\pi_{k-2}(F_{k,r})=d(\mathcal P_{k,r})=\frac{2(r-1)}{rk^k}.
\]
For fixed $k$, these values are strictly increasing, since the difference between the terms with indices $r+1$ and $r$ is $2/(k^kr(r+1))>0$. Their limit is $2/k^k$. Hence $2/k^k$ is an accumulation point of $\Piuk$, completing the proof.
\end{proof}

\section{A sharp threshold and matching constructions}\label{sec:obstructions}

We first prove the lower bound in Theorem \ref{thm:threshold}. 
Fix $i<j$, and take pairwise disjoint color classes
\(
A,B,C,D_\ell
\)
with $|A|=|B|=|C|=2$ and $|D_\ell|=3$ for $\ell\in[k]\setminus\{i,j\}$. 
The color set has size $3k$. Define a palette $\mathcal V_{i,j}$ by requiring
\[
x_i\in A\cup C,
\quad x_j\in B\cup C,
\quad \text{and} \quad x_\ell\in D_\ell\quad \text{for } \ell\ne i,j,
\]
except that we do not allow both $x_i$ and $x_j$ to belong to $C$. Thus the possible pairs in coordinates $i,j$ form
\[
(A\times B)\cup(A\times C)\cup(C\times B).
\]
Each product has four elements, and hence
\begin{equation}\label{eq:V-density}
d(\mathcal V_{i,j})=
\frac{12\cdot3^{k-2}}{(3k)^k}
=\frac4{3k^k}.
\end{equation}

\begin{lemma}\label{lem:V-implies}
Every $k$-graph that admits $\mathcal V_{i,j}$ satisfies $\Sp$ for $\{i,j\}$.
\end{lemma}

\begin{proof}
Fix a witnessing order and coloring. Let $F^1_{i,j}$ contain the edges whose color in coordinate $j$ belongs to $B$, and let $F^2_{i,j}$ contain the edges whose color in coordinate $j$ belongs to $C$. These two subgraphs partition the edge set because the only possible pairs in coordinates $i,j$ are the three products displayed above.

In $F^1_{i,j}$, colors from $A\cup C$ occur only in coordinate $i$, colors from $B$ occur only in coordinate $j$, and colors from $D_\ell$ occur only in coordinate $\ell$. Thus the witnessing order is vanishing for $F^1_{i,j}$. In $F^2_{i,j}$, the pair $C\times C$ is forbidden, so the color in coordinate $i$ belongs to $A$. Hence colors from $A$ occur only in coordinate $i$, colors from $C$ occur only in coordinate $j$, and colors from $D_\ell$ occur only in coordinate $\ell$. The same order is therefore vanishing for $F^2_{i,j}$.

Let a facet lie in an edge of each part. A color from $A$ gives type $i$ in both parts, and a color from $D_\ell$ gives type $\ell$ in both parts. A color from $C$ gives type $i$ in $F^1_{i,j}$ and type $j$ in $F^2_{i,j}$. A color from $B$ cannot occur in $F^2_{i,j}$. Thus the compatibility condition in $\Sp$ holds.
\end{proof}

\begin{proof}[\bfseries{Proof of Theorem \ref{thm:threshold}}]
Let $F$ fail $\Sp$. Then it fails the condition for some pair $\{i,j\}$. 
By Lemma \ref{lem:V-implies}, $F$ does not admit $\mathcal V_{i,j}$. Equations \eqref{eq:palette-lower} and \eqref{eq:V-density} yield
\[
\pi_{k-2}(F)\ge\frac4{3k^k}.
\]

For the reverse inequality, take the graph $F_{k,3}$ constructed in the proof of Theorem \ref{thm:sequence}. It does not admit $\mathcal P_{k,3}$ and satisfies $\pi_{k-2}(F_{k,3})=4/(3k^k)$. We claim that
\begin{equation}\label{eq:W-to-P}
\mathcal W_{1,k}\longrightarrow\mathcal P_{k,3}.
\end{equation}
Choose colors $\alpha_s\in A_s$ for $s\in[3]$ and $\delta_\ell\in D_\ell$ for $2\le\ell\le k-1$. Define
\[
a_1\mapsto\alpha_1,
\quad \omega\mapsto\alpha_2,
\quad a_k\mapsto\alpha_3,
\quad \text{and} \quad 
a_\ell\mapsto\delta_\ell\quad \text{for all } 2\le\ell\le k-1.
\]
The pairs in coordinates $1$ and $k$ in the images of the three admissible tuples of $\mathcal W_{1,k}$ lie in $A_1\times A_3$, $A_2\times A_3$, and $A_1\times A_2$, respectively. They are all admissible, so \eqref{eq:W-to-P} holds.

If $F_{k,3}$ satisfied $\Sp$ for $\{1,k\}$, Proposition \ref{prop:W-equivalence} would imply that it admits $\mathcal W_{1,k}$. Equation \eqref{eq:W-to-P} would then imply that it admits $\mathcal P_{k,3}$, a contradiction. Hence $F_{k,3}$ fails $\Sp$ and has density $4/(3k^k)$. This completes the proof.
\end{proof}

The new threshold strictly improves the earlier lower bound:
\[
\frac1{k^k}<\frac3{(k+1)^k}<\frac4{3k^k} \quad
\text{for all } k\ge3.
\]
The first inequality follows from $(1+1/k)^k<e<3$. For the second, the binomial theorem gives
\[
\left(1+\frac1k\right)^k
>2+\frac{k-1}{2k}
\ge\frac73>\frac94.
\]
Therefore, any possible example in Problem~7.4 of Lin, Wang and Zhou with density $3/(k+1)^k$ must satisfy $\Sp$.

We now prove Theorem \ref{thm:matching}. Let $\mathcal M$ be a matching of size $m$ satisfying the assumptions of Theorem~\ref{thm:matching}, and let
\[
V(\mathcal M)=\bigcup_{h\in\mathcal M}h
\]
be the set of coordinates covered by the matching. 
For every $h=\{p,q\}\in\mathcal M$, with $p<q$, take two colors $x_h,y_h$. 
These two colors are used in both coordinates $p$ and $q$. 
For every coordinate $s\in[k]\setminus V(\mathcal M)$, take a new color $z_s$, used only in coordinate $s$. 
All chosen colors are pairwise distinct. Define the target palette $\mathcal B_{\mathcal M}$ by allowing, independently for every $h=\{p,q\}\in\mathcal M$, either
\[
(x_h,y_h)\quad\text{or}\quad(y_h,x_h)
\]
in coordinates $p,q$, and by requiring the color $z_s$ in every unmatched coordinate $s$. The palette has $2m+(k-2m)=k$ colors and $2^m$ admissible tuples. Hence
\begin{equation}\label{eq:B-density}
d(\mathcal B_{\mathcal M})=\frac{2^m}{k^k}.
\end{equation}
Since $\rho(\mathcal M)=\mathcal M$, reversing the coordinates sends the block indexed by $h$ to the block indexed by $\rho(h)$. The map
\[
x_h\mapsto x_{\rho(h)},\quad y_h\mapsto y_{\rho(h)},\quad \text{and} \quad z_s\mapsto z_{\rho(s)}
\]
is therefore an isomorphism from $\rev(\mathcal B_{\mathcal M})$ to $\mathcal B_{\mathcal M}$.

We next introduce the oriented cycles used in the auxiliary palettes. Let $r$ be a positive integer. The digraph $J_r$ is the orientation of the cycle on $v_0,v_1,\ldots,v_{3r-1}$ in which, for every $0\le a<r$, the three consecutive edges are oriented as
\begin{equation}\label{eq:odd-orientation}
v_{3a}\to v_{3a+1},
\quad v_{3a+1}\to v_{3a+2},
\quad \text{and} \quad v_{3a+3}\to v_{3a+2},
\end{equation}
where the subscripts are taken modulo $3r$. Let $T_3$ be the transitive tournament on $\{0,1,2\}$, and let $\overrightarrow{C}_s$ be the directed cycle on $\mathbb Z_s$ whose arcs are $b\to b+1$.

\begin{lemma}\label{lem:odd-cycle-maps}
For every $r\ge1$, there is a homomorphism $J_r\to T_3$. If $s\ge2$ divides $r$, then $J_r\to\overrightarrow{C}_s$. Moreover, the underlying graph of $J_r$ is an odd cycle whenever $r$ is odd.
\end{lemma}

\begin{proof}
Map $v_{3a},v_{3a+1},v_{3a+2}$ to $0,1,2$, respectively. The three arcs in \eqref{eq:odd-orientation} are then mapped to $0\to1$, $1\to2$, and $0\to2$, so this is a homomorphism to $T_3$.

Now suppose that $s\mid r$. For every $0\le a<r$, define, with all values taken modulo $s$,
\[
f(v_{3a})=a,\quad f(v_{3a+1})=a+1,\quad \text{and} \quad f(v_{3a+2})=a+2.
\]
Since $f(v_{3r})=r=0=f(v_0)$ in $\mathbb Z_s$, the definition is consistent at the last edge of the cycle. Every arc in \eqref{eq:odd-orientation} is mapped to an arc of $\overrightarrow{C}_s$, and hence $J_r\to\overrightarrow{C}_s$. Finally, the underlying graph of $J_r$ is a cycle of length $3r$, which is odd when $r$ is odd.
\end{proof}

For every $h=\{p,q\}\in\mathcal M$, with $p<q$, let
\[
d_h=q-p+1
\quad \text{and} \quad
r_h=\operatorname{lcm}(3,d_h).
\]
The assumption that $q-p$ is even implies that $d_h$ is odd. Hence $r_h$ is odd and is divisible by both $3$ and $d_h$. Define
\[
\mathcal Q_h=\mathcal Q(J_{r_h};p,q).
\]
By Lemma \ref{lem:odd-cycle-maps}, $J_{r_h}$ admits homomorphisms to $T_3$, $\overrightarrow{C}_3$, and $\overrightarrow{C}_{d_h}$, and its underlying graph is an odd cycle. The map $0\mapsto a_p$, $1\mapsto\omega$, $2\mapsto a_q$ sends the three arcs of $T_3$ to the three pairs occurring in coordinates $p,q$ of the admissible tuples of $\mathcal W_{p,q}$. Since each additional color of $\mathcal Q_h$ occurs in only one coordinate of one tuple, this map extends to a palette homomorphism
\begin{equation}\label{eq:Q-to-W}
\mathcal Q_h\longrightarrow\mathcal W_{p,q}.
\end{equation}

For every $g=\{p,q\}\in\binom{[k]}2\setminus\mathcal M$, use the switch palette $\mathcal R_{p,q}$. Let $\mathfrak A_{\mathcal M}$ be the family consisting of these palettes and $\mathcal Q_h$, $h\in\mathcal M$.

\begin{lemma}\label{lem:upper-matching}
No auxiliary palette admits a homomorphism to $\mathcal B_{\mathcal M}$ or to $\rev(\mathcal B_{\mathcal M})$. Moreover, if a $k$-palette $\Palt$ satisfies $d(\Palt)>2^m/k^k$, then one of the auxiliary palettes admits a homomorphism to $\Palt$ or to $\rev(\Palt)$.
\end{lemma}

\begin{proof}
We first prove the non-homomorphism statement. Let $g=\{p,q\}\notin\mathcal M$. In $\mathcal B_{\mathcal M}$, the supports of two distinct coordinates meet exactly when their coordinate pair belongs to $\mathcal M$. Thus the special color of $\mathcal R_{p,q}$ cannot be mapped simultaneously into the supports of coordinates $p$ and $q$. Hence $\mathcal R_{p,q}\not\to\mathcal B_{\mathcal M}$. Since $\rho(\mathcal M)=\mathcal M$, the same argument applies to the reverse palette.

Now let $h=\{p,q\}\in\mathcal M$. Suppose that $\mathcal Q_h\to\mathcal B_{\mathcal M}$. Every vertex color of $J_{r_h}$ is then mapped to one of $x_h,y_h$, and the endpoints of every arc are mapped to different colors. This would give a proper two-coloring of the underlying cycle of $J_{r_h}$, which is impossible because $r_h$ is odd. Hence $\mathcal Q_h\not\to\mathcal B_{\mathcal M}$. In the reverse target, coordinates $p,q$ carry the two-color block indexed by $\rho(h)\in\mathcal M$, so the same argument applies.

It remains to prove the density statement. Suppose for a contradiction that no auxiliary palette admits a homomorphism to $\Palt=(\C,\T)$ or to $\rev(\Palt)$. Write $C_i=C_i(\Palt)$ for every $i\in[k]$. The absence of a homomorphism from $\mathcal R_{p,q}$ gives
\begin{equation}\label{eq:matching-support}
C_p\cap C_q=\varnothing
\quad\text{whenever }\{p,q\}\notin\mathcal M.
\end{equation}
Indeed, a color in the intersection could be used as the image of the special color, while the remaining colors, each of which occurs once, could be mapped to two witnessing tuples.

For every $h=\{p,q\}\in\mathcal M$, let $U_h=C_p\cup C_q$. Since $\mathcal M$ is a matching, \eqref{eq:matching-support} shows that all sets $U_h$, together with the supports $C_s$ for $s\notin V(\mathcal M)$, are pairwise disjoint. Therefore
\begin{equation}\label{eq:block-sum}
\sum_{h\in\mathcal M}|U_h|+
\sum_{s\notin V(\mathcal M)}|C_s|
\le|\C|.
\end{equation}

Fix $h=\{p,q\}\in\mathcal M$ with $p<q$. Define a digraph $L_h$ on $U_h$, allowing loops at this point, by placing an arc $x\to y$ whenever some admissible tuple of $\Palt$ has $x$ in coordinate $p$ and $y$ in coordinate $q$.

We first observe that there is no homomorphism $J_{r_h}\to L_h$. Indeed, if $f:J_{r_h}\to L_h$ were such a homomorphism, then for every arc $uv$ of $J_{r_h}$ we could choose an admissible tuple of $\Palt$ witnessing the arc $f(u)\to f(v)$. Mapping the vertex colors of $\mathcal Q_h$ according to $f$ and mapping the additional colors to the corresponding coordinates of the chosen tuples would give a homomorphism $\mathcal Q_h\to\Palt$, a contradiction.

In particular, $L_h$ has no loop, since a constant map would otherwise give $J_{r_h}\to L_h$. Suppose that $|A(L_h)|>|U_h|^2/2$. Let $G_h$ be the simple graph on $U_h$ in which $xy$ is an edge if $L_h$ contains at least one of the arcs $x\to y$ and $y\to x$. Since every edge of $G_h$ accounts for at most two arcs of $L_h$,
\[
e(G_h)\ge\frac{|A(L_h)|}{2}>\frac{|U_h|^2}{4}.
\]
By Mantel's theorem~\cite{Mantel1907}, $G_h$ contains a triangle. 
Choosing one available direction on each of its three edges gives either a transitive tournament $T_3$ or a directed triangle $\overrightarrow{C}_3$ inside $L_h$. 
By Lemma \ref{lem:odd-cycle-maps}, $J_{r_h}\to T_3$, and $J_{r_h}\to\overrightarrow{C}_3$ because $3\mid r_h$. In both cases we obtain $J_{r_h}\to L_h$, a contradiction. Consequently,
\begin{equation}\label{eq:matching-arcs}
|A(L_h)|\le\frac{|U_h|^2}{2}.
\end{equation}

Every admissible tuple determines one arc of $L_h$ for each $h\in\mathcal M$ and one color from $C_s$ for each unmatched coordinate $s$. It follows that
\[
|\T|\le\prod_{h\in\mathcal M}|A(L_h)|
\prod_{s\notin V(\mathcal M)}|C_s|.
\]
Using \eqref{eq:matching-arcs}, we obtain
\[
|\T|\le2^m
\prod_{h\in\mathcal M}\left(\frac{|U_h|}{2}\right)^2
\prod_{s\notin V(\mathcal M)}|C_s|.
\]
The product after $2^m$ has exactly $2m+(k-2m)=k$ non-negative factors, and their sum is the left-hand side of \eqref{eq:block-sum}. The arithmetic--geometric mean inequality therefore gives
\[
|\T|\le2^m\left(\frac{|\C|}{k}\right)^k.
\]
Dividing by $|\C|^k$ yields $d(\Palt)\le2^m/k^k$, contrary to the assumption. This proves the lemma.
\end{proof}

The next lemma verifies the mixed-product condition.

\begin{lemma}\label{lem:pairwise-transfer-matching}
For every $\mathcal A\in\mathfrak A_{\mathcal M}$, there is no homomorphism from
\[
\mathcal A\times
\prod_{\mathcal A'\in\mathfrak A_{\mathcal M}\setminus\{\mathcal A\}}
(\mathcal A')^{\Sym}
\]
to $\mathcal B_{\mathcal M}$ or to $\rev(\mathcal B_{\mathcal M})$.
\end{lemma}

\begin{proof}
By Lemma \ref{lem:upper-matching}, no auxiliary palette admits a homomorphism to the target or to its reverse. We now connect the auxiliary palettes through symmetrization.

Let $\mathcal A,\mathcal A'\in\mathfrak A_{\mathcal M}$ be distinct. 
Suppose first that $\mathcal A$ is a switch palette. 
If $\mathcal A'$ is another switch palette, its special color occurs in two admissible tuples. 
If $\mathcal A'=\mathcal Q_h$, every vertex color of the underlying cycle occurs in two distinct tuple-coordinate positions. 
In either case, Lemma \ref{lem:repeated-to-sym} gives
\[
\mathcal A\longrightarrow(\mathcal A')^{\Sym}.
\]

It remains to consider $\mathcal A=\mathcal Q_h$, where $h=\{p,q\}$ and $p<q$. 
Every admissible tuple of $\mathcal A'$ has pairwise distinct entries. 
By Lemma \ref{lem:cycle}, the coordinate digraph of $(\mathcal A')^{\Sym}$ in positions $p,q$ contains a directed cycle of length $d_h=q-p+1$. 
Since $d_h\mid r_h$, Lemma \ref{lem:odd-cycle-maps} gives $J_{r_h}\to\overrightarrow{C}_{d_h}$. Lemma~\ref{lem:graph-to-sym} therefore yields
\[
\mathcal Q_h\longrightarrow(\mathcal A')^{\Sym}.
\]
Thus every auxiliary palette admits a homomorphism to the symmetrization of every other auxiliary palette.

Fix $\mathcal A\in\mathfrak A_{\mathcal M}$. 
Use the identity homomorphism in the $\mathcal A$-factor and the homomorphisms just proved in all other factors. 
By the definition of a palette product, these maps give a homomorphism from $\mathcal A$ to the mixed product in the statement. 
If that product admitted a homomorphism to $\mathcal B_{\mathcal M}$ or to its reverse, 
composition would contradict Lemma \ref{lem:upper-matching}. This proves the lemma.
\end{proof}

\begin{proof}[\bfseries{Proof of Theorem \ref{thm:matching}}]
Apply Theorem \ref{thm:realization} to the target palette $\mathcal B_{\mathcal M}$ and the auxiliary palettes in $\mathfrak A_{\mathcal M}$. 
Equation \eqref{eq:B-density} states that $d(\mathcal B_{\mathcal M})=2^m/k^k$. 
Lemma~\ref{lem:pairwise-transfer-matching} verifies condition \ref{E1}, while Lemma \ref{lem:upper-matching} verifies condition \ref{E2}. Therefore there exists a finite $k$-graph $F_{\mathcal M}$ that admits every auxiliary palette, does not admit $\mathcal B_{\mathcal M}$, and satisfies
\[
\pi_{k-2}(F_{\mathcal M})=d(\mathcal B_{\mathcal M})=\frac{2^m}{k^k}.
\]

It remains to verify $\Sp$. If $g=\{p,q\}\notin\mathcal M$, then $F_{\mathcal M}$ admits $\mathcal R_{p,q}$, and \eqref{eq:R-to-W} shows that it admits $\mathcal W_{p,q}$. If $h=\{p,q\}\in\mathcal M$, then $F_{\mathcal M}$ admits $\mathcal Q_h$, and \eqref{eq:Q-to-W} shows that it admits $\mathcal W_{p,q}$. Proposition~\ref{prop:W-equivalence} now implies that $F_{\mathcal M}$ satisfies $\Sp$ for every coordinate pair.

Finally, $\pi_{k-2}(F_{\mathcal M})=2^m/k^k>0$. The zero-density characterization recalled in Section~\ref{sec:Preliminaries} shows that $F_{\mathcal M}$ has no vanishing order. This completes the proof.
\end{proof}

\section{Concluding remarks}\label{sec:concluding-remarks}

The extension parameters do not increase with $t$, while the classes of $t$-layered hypergraphs increase with $t$.

\begin{proposition}\label{prop:monotone}
For every $k$-palette $\Palt$,
\[
\delta_0(\Palt)\ge\delta_1(\Palt)\ge\cdots\ge\delta_{k-1}(\Palt).
\]
For every $k$-graph $F$, we also have
\[
\pi_0^{\ext}(F)\ge\pi_1^{\ext}(F)\ge\cdots\ge\pi_{k-1}^{\ext}(F),
\]
Every $t$-layered $k$-graph is $(t+1)$-layered for $0\le t\le k-2$.
\end{proposition}

\begin{proof}
Fix $0\le t\le k-2$, a set $I\in\binom{[k]}t$, and $\boldsymbol c\in\C^I$. Choose $j\in[k]\setminus I$. For $a\in\C$, let $\boldsymbol c_a\in\C^{I\cup\{j\}}$ extend $\boldsymbol c$ by $\boldsymbol c_a(j)=a$. Grouping the completions according to the color in coordinate $j$ gives
\[
\frac{|N_I^{\Palt}(\boldsymbol c)|}{|\C|^{k-t}}
=
\frac1{|\C|}\sum_{a\in\C}
\frac{|N_{I\cup\{j\}}^{\Palt}(\boldsymbol c_a)|}{|\C|^{k-t-1}}
\ge\delta_{t+1}(\Palt).
\]
Taking the minimum over $I$ and $\boldsymbol c$ yields $\delta_t(\Palt)\ge\delta_{t+1}(\Palt)$. The inequality for $\pi_t^{\ext}$ follows because both suprema are taken over the same palettes not admitted by $F$.

Finally, suppose that a vector order satisfies \ref{D1} and \ref{D2} for some $t\le k-2$. In each vector, move any one head coordinate into its tail set. The remaining head values are still pairwise distinct, are different from the new tail set, and have not appeared in an earlier vector. Thus the same vector order satisfies \ref{D1} and \ref{D2} with $t+1$ in place of $t$, proving the final assertion.
\end{proof}

We finish with an application to cherry Tur\'an density. Let $H=(V,E)$ be a $3$-graph and let $A,B\subseteq V^2$ be sets of ordered pairs. Define
\[
K_{\mathrm{ch}}(A,B)=\{(x,y,z)\in V^3:(x,y)\in A,\ (x,z)\in B\}
\]
and
\[
e_H(A,B)=\bigl|\{(x,y,z)\in K_{\mathrm{ch}}(A,B):\{x,y,z\}\in E\}\bigr|.
\]
The graph $H$ is \emph{$(d,\mu,\mathrm{ch})$-dense} if
\[
e_H(A,B)\ge d|K_{\mathrm{ch}}(A,B)|-\mu|V|^3
\]
for all $A,B\subseteq V^2$. For a $3$-graph $F$, define
\begin{align*}
\pi_{\mathrm{ch}}(F)=\sup\bigl\{d\in[0,1]:{}&\text{for every }\mu>0\text{ and }n_0\in\N,\text{ there is an $F$-free}\notag\\[-1mm]
&\text{$(d,\mu,\mathrm{ch})$-dense $3$-graph $H$ with }|V(H)|\ge n_0\bigr\}.
\end{align*}
For $3$-palettes, $\delta_2$ is Lamaison's minimum palette codegree. His random palette construction~\cite{Lamaison2024} therefore gives
\begin{equation}\label{eq:cherry-palette-lower}
\pi_{\mathrm{ch}}(F)\ge\sup\{\delta_2(\Palt):F\text{ does not admit }\Palt\}=\pi_2^{\ext}(F).
\end{equation}

\begin{corollary}\label{cor:cherry}
If a $3$-graph $F$ satisfies $\pi_{\mathrm{ch}}(F)=0$, then $F$ is max-layered. More precisely, if $F$ is not max-layered, $M=|\partial F|$, and $q=2^{M+1}-1$, then
\[
\pi_{\mathrm{ch}}(F)\ge(3q)^{-q}>0.
\]
\end{corollary}

\begin{proof}
Suppose that $F$ is not max-layered. 
By Proposition \ref{prop:endpoints}(iii), it is not $2$-layered. 
Apply the explicit lower bound in Theorem \ref{thm:main-intro} with $k=3$ and $t=2$. Here
\[
b_{3,2}=(3-2)\binom32=3,
\quad
q_{2,M}=1+2+\cdots+2^M=2^{M+1}-1=q,
\]
and $(k-t)q=q$. Hence $\pi_2^{\ext}(F)\ge(3q)^{-q}$. 
Equation \eqref{eq:cherry-palette-lower} gives the same lower bound for $\pi_{\mathrm{ch}}(F)$. Therefore, if $\pi_{\mathrm{ch}}(F)=0$, then $F$ must be max-layered.
\end{proof}

Positive coordinate-extension degree characterizes the classes of $t$-layered hypergraphs. Ordinary palette density, together with the realization theorem, yields the exact values constructed here. 

Apart from trying to determine further elements of $\Piuk$, one may ask which values are its accumulation points. It is not known whether $2/k^k$ is the first accumulation point above $k^{-k}$ for single forbidden $k$-graphs, 
or whether exact densities can converge to $k^{-k}$ from above. 

\section*{Acknowledgments}
\noindent This research was supported by the National Key R\&D Program of China under grant number 2024YFA1013900, the NSFC under grant number 12471327, and the China Postdoctoral Science Foundation under grant number 2026M793375.

\section*{Declaration on the use of AI}
At the beginning of this project, the author(s) had only constructed Theorem~\ref{thm:matching} for the case $k=3$. 
The author(s) then asked AI tools to work on the problem directly, to see whether the result could be generalized to $k \geq 3$. 
The initial attempts mainly consisted of trying to optimize methods from previous work and did not lead to the result. 
After the author(s) provided the construction for how the case $k=3$ should be done, and with this strategy in place, the AI tools quickly produced an argument yielding the theorem. 
In addition, the authors used AI to assist in discussing proof strategies, checking proofs, and improving exposition. After using this tool, the authors reviewed and edited the content as needed, independently verified all mathematical arguments, and take full responsibility for the content of the article.

\end{document}